\documentclass[11pt]{article}
\usepackage{etex}
\usepackage{amsmath,amssymb,amsfonts,amsthm}
\usepackage{a4wide}
\usepackage{color}
\usepackage{verbatim}
\usepackage[]{graphics}
\usepackage{graphicx,inputenc}
\usepackage{subfigure}
\usepackage[justification=centering]{caption}
\usepackage{pictex}
\usepackage{wrapfig}
\usepackage{multirow}
\usepackage[T1]{fontenc}
\usepackage{authblk}
\usepackage[textwidth=6.0in,textheight=9in]{geometry}

\usepackage[colorlinks,linktocpage,linkcolor=blue]{hyperref}
\usepackage{fancyhdr}
\usepackage{ulem}

\newtheoremstyle{exampstyle}
  {\topsep} 
  {\topsep} 
  {\itshape} 
  {} 
  {\bfseries} 
  {.} 
  {.5em} 
  {} 

\theoremstyle{exampstyle}

\numberwithin{equation}{section}
\newtheorem{lemma}{Lemma}[section]
\newtheorem{theorem}{Theorem}[section]
\newtheorem{assumption}{Assumption}[section]
\newtheorem{remark}{Remark}[section]
\newtheorem{definition}{Definition}[section]

\newtheorem{corollary}[lemma]{Corollary}

\allowdisplaybreaks

\usepackage{parskip}
\let\oldref\ref
\renewcommand{\ref}[1]{(\oldref{#1})}  
\renewcommand{\eqref}[1]{(\oldref{#1})}

\newbox\boxaddrone \newbox\boxaddrtwo

\def\N+{n\in\mathbb{N}^{+}}

\def\A{\mathcal{A}}
\def\l{\langle}
\def\rd{\rangle_{L^2(D)}}

\def\E{\mathbb{E}}

\def\A{\mathcal{A}}

\def\I{I_t^\alpha}
\def\o{\omega}

\def\Nua{\nu_{\mathcal A}}

\begin{document}

\title{\large\textbf{Well-posedness of the stochastic time-fractional diffusion and wave equations and inverse random source problems}}
\author[1]{Matti Lassas\thanks{matti.lassas@helsinki.fi}}
\author[2]{Zhiyuan Li\thanks{lizhiyuan@nbu.edu.cn}}
\author[3]{Zhidong Zhang\thanks{zhangzhidong@mail.sysu.edu.cn}}
\affil[1]{\normalsize{Department of Mathematics and Statistics, 
University of Helsinki, Finland}}
\affil[2]{\normalsize{School of Mathematics and Statistics, Ningbo University, China}}
\affil[3]{\normalsize{School of Mathematics (Zhuhai), Sun Yat-sen University, China}}

\maketitle

\begin{abstract}
\noindent In this paper, we are concerned with the stochastic time-fractional diffusion-wave equations in a Hilbert space. The main objective of this paper is to establish properties of the stochastic weak solutions of the initial-boundary value problem, such as the existence, uniqueness and regularity estimates. Moreover, we apply the obtained theories to an inverse source problem. The uniqueness of this inverse problem under the boundary measurements is proved. \\


\end{abstract}

\section{Introduction.}
\label{sec-intro-distri}

\subsection{Mathematical model.}
In this work, we consider the following initial-boundary value problem for the stochastic time-fractional diffusion-wave equation (STFDE): 
\begin{equation}\label{equ-gov}
\left\{
\begin{aligned}
\partial_t^\alpha u(x,t,\o) + \mathcal A u(x,t,\o) &= F(x,t,\o),&&\quad (x,t,\o)\in D\times(0,\infty)\times\Omega,\\
u(x,0,\o)&=u_0(x,\o), &&\quad (x,\o)\in D\times\Omega,\\
\partial_tu(x,0,\o)&=u_1(x,\o), &&\quad (x,\o)\in D\times\Omega,\ \text{if}\ \alpha\in(1,2),\\
u(x,t,\o) &= 0, &&\quad (x,t,\o)\in \partial D\times(0,\infty)\times\Omega. 
\end{aligned}
\right.
\end{equation}
Here $D$ is an open bounded domain in $\mathbb R^d$ with a $C^\infty$ smooth boundary $\partial D$, and $(\Omega,\mathcal{F},P)$ is a complete probability space, which will be introduced later. We define the source $F$ in terms of the Ito integral \cite{Bernt2003stochastic} as 
\begin{equation}\label{F}
 \begin{aligned}
  F(x,t,\o)&=I_t^\delta \Big( f_1(x,t,\o)+f_2(x,t,\o)\frac{dB(t)}{dt} \Big),
 \end{aligned}
\end{equation}
where $\delta$ satisfies 
\begin{equation}\label{condition_delta}
 \delta\in[0,1/2),\ \alpha+\delta> 1/2. 
\end{equation}
The notation $\Gamma(\cdot)$ means the Gamma function, $I_t^\delta$ is the Riemann-Liouville fractional integral operator, which will be introduced later. The operator $\mathcal A$ is a symmetric uniformly elliptic operator with the homogeneous Dirichlet boundary condition. It is defined as for $u\in H^2(D) \cap H_0^1(D)$,  
$$
\mathcal A u(x)
= -\sum_{i,j=1}^d \frac{\partial}{\partial x_i}\Big(a_{ij}(x)
\frac{\partial}{\partial x_j}u(x)\Big)+c(x)u(x),
\quad x\in D,
$$
where $c$, $a_{ij} \in C^1(\overline D)$ satisfy $c\ge0$ and $a_{ij}=a_{ji}$ for $1\le i,j\le d$. Moreover there exists a constant $a_0>0$ such that
$$
a_0 \sum_{j=1}^d \xi_j^2 
\le \sum_{j,k=1}^d a_{jk}(x) \xi_j\xi_k, 
\quad x\in\overline{D},\ \xi\in\mathbb R^d.
$$
In the problem \eqref{equ-gov}, $\partial_t^\alpha$ is the 
Djrbashyan-Caputo derivative of order $\alpha$, given as  
$$
\partial_t^\alpha \varphi(t):=\frac{1}{\Gamma(n-\alpha)}\int_0^t 
  (t-\tau)^{n-\alpha-1} \varphi^{(n)}(\tau)\ d\tau,\quad n-1<\alpha <n,\quad n\in\mathbb{N}^+,
$$
where $\varphi^{(n)}$ means the $n$-th derivative of $\varphi$. We set $\alpha\in(0,1)\cup(1,2)$ in this work.

\subsection{Background and literature.}

Recently, some anomalous diffusion processes whose mean square displacement with non-Fickian growth were found in more and more application areas. These include porous media (Levy and Berkowitz \cite{LeBer03}), environmental engineering (Sokolov \cite{Sok12}, Hatano and Hatano \cite{Ha98}), chemistry (Uchaikin \cite{Uch13}), plasma turbulence (Del Castillo Negrete et al.  \cite{DCL04, DCL05}), biological systems (Kneller \cite{Kn15}), fractal geometry (Nigmatullin \cite{Nig86}) and various other physical models.

One of the approaches for modeling such processes is to employ the diffusion-wave equations of time-fractional order derivative. For instance, Metzler and Klafter \cite{MeKl00} derived the time-fractional diffusion-wave equation in the framework of the continuous time random walk, and Magdziarz et al. \cite{MWW07} used the time-changed Langevin equations to deduce the temporal fractional diffusion-advection equation. Also it is well known that the fractional diffusion-wave equation performs well in describing the anomalous phenomena in such as a highly heterogeneous aquifer, complex viscoelastic materials and underground environmental problems. The corresponding works can be found in Adams and Gelhar \cite{adams1992field}, Barkai  \cite{barkai2001fractional}, Giona, Cerbelli and Roman\cite{giona1992fractional}, Hatano and Hatano \cite{hatano1998dispersive}, Metzler, Barkai and Klafter \cite{metzler1999anomalous}, Metzler, Klafter and Sokolov \cite{metzler1998anomalous} and the references therein. As a result, the field of fractional diffusion-wave equations attracts great attention from scholars in  mathematics. We refer to Kubica and Yamamoto \cite{kubica2018initial}, Sakamoto and Yamamoto \cite{Sakamoto2013Inverse} and Zacher \cite{zacher2005maximal} for the well-posedness of the forward problem for the fractional diffusion-wave equation; Li, Liu and Yamamoto \cite{li2015initial} for the further properties of the solutions; Vergara and Zacher \cite{vergara2015optimal} for the asymptotic behaviors of the solution;  Luchko and Yamamoto \cite{luchko2017maximum, luchko2019maximum} for the (strong) maximum principle. Furthermore, we  refer to Chen and Stynes \cite{chen2019error} and Jin, Lazarov and Zhou \cite{jin2016analysis} for the numerical treatments for the forward problems of fractional diffusion-wave equations. 

In many practical physical system, the source admits a non-ignorable random disturbance, such as the Brownian random processes. Therefore, their governing equations are stochastic differential equations (SDEs) involving deterministic and uncertain sources. For the introductions of SDEs, we refer Evans \cite{evans2012introduction} and Ishimaru \cite{ishimaru1978wave}. On the other hand, 
unlike the extensive investigation of the deterministic time-fractional differential equations from both the mathematical and numerical aspects, the research in the stochastic case is relatively scarce due to the presence of randomness. Nonetheless, some interesting works have
been exhibited. 
By using the fundamental solution argument and H-fox functions, Chen and Hu \cite{chen2022holder} and Mijena and Nane \cite{mijena2015space} discussed the initial-value problem for stochastic fractional diffusion equations in the whole domain $\mathbb R^d$, where the existence and uniqueness of solutions and the H\"older regularity estimates were established under suitable assumptions. Unfortunately, these approaches could not be extended directly to the bounded domain case since the explicit Green function is not available for the general bounded domain.
Utilizing the eigenfunction expansion argument and the Mittag-Leffler functions, Liu, Wen and Zhang \cite{liu2019reconstruction} and Niu, Helin and Zhang \cite{niu2020inverse} established several priori regularity estimates by assuming the existence of the weak solution. He and Peng \cite{he2019approximate} proved the unique existence of the mild solution for the STFDEs by the semigroup theory but without giving the further regularity estimates. For the well-posedness result of the STFDEs with fractional Brownian motion, we refer to Feng, Li and Wang \cite{feng2020inverse} and the references therein.
For the numerical treatments for the STFDEs, several novel and efficient algorithms have been developed in Li, Wang and Deng \cite{li2017galerkin}, Zou \cite{zou2018galerkin} and Zou, Atangana and Zhou \cite{zou2018error}.

\subsection{Main results and outline.}
Before stating the main theorems, we need to introduce the assumptions of the source $F$ and initial conditions $u_0, u_1$. 
\begin{assumption}\label{assumption}
For $f_j: D\times(0,\infty)\times\Omega\to \mathbb{R},\ j=1,2$, we assume that:  
\begin{itemize}
 \item[$(1)$] $(x,t,\omega)\mapsto f_j(x,t,\omega),\ j=1,2$ are  
 $\mathcal D\times\mathcal{B}\times \mathcal{F}$-measurable;
 \item[$(2)$] $(x,\omega)\mapsto f_j(x,t,\omega), \ j=1,2$ 
 are $\mathcal D\times\mathcal{F}_t$-measurable for all $t\in[0,\infty)$.  
\end{itemize}
Furthermore, the initial conditions $u_0: \Omega\to H_0^1(D),\ u_1:\Omega\to L^2(D)$ and the source terms  
$f_j: [0,\infty)\to L^2(D\times\Omega),\ j=1,2$ satisfy the following regularity estimates: 
 \begin{itemize}
 \item[$(3)$] $u_0\in L^2(\Omega;H^1(D)),\ u_1\in L^2(D\times\Omega)$;
  \item[$(4)$] $f_j\in C([0,\infty);L^2(D\times\Omega))$ and $\|f_j\|_{C([0,\infty);L^2(D\times\Omega))}<\infty$ for $j=1,2$;
  \item[$(5)$] For $T\in(0,\infty)$ and $j=1,2$, $f_j\in L^2(D\times(0,T)\times \Omega)$.
 \end{itemize}
\end{assumption}

Next we introduce the definition of the Riemann-Liouville fractional integral $I_t^\alpha$. 
\begin{definition}[Riemann-Liouville fractional integral]\label{I_alpha}
Letting $\alpha>0$, we define the Riemann-Liouville fractional integral operator $I_t^\alpha$ on the function space $L_{\rm loc}^1([0,\infty))$ as
 \begin{equation*}
  I_t^\alpha \varphi (t)=\Gamma(\alpha)^{-1}\int_0^t 
  (t-\tau)^{\alpha-1} \varphi(\tau)\ d\tau
 \end{equation*}
 for any $\varphi\in L_{\rm loc}^1([0,\infty))$, with the convention 
 $I_t^0 \varphi(t)=\varphi(t)$. 
For a random variable $f(t,\o)\frac{dB(t)}{dt}$ with $f\in C([0,\infty);L^2(\Omega))$, we define its Riemann-Liouville fractional integral $I_t^\alpha \Big(f(t,\o)\frac{dB(t)}{dt}\Big)$ as
 \begin{equation*}
  I_t^\alpha \Big(f(t,\o)\frac{dB(t)}{dt}\Big)=\Gamma(\alpha)^{-1}\int_0^t (t-\tau)^{\alpha-1}f(\tau,\o)\ dB(\tau).
 \end{equation*}
 \end{definition}
 From direct calculation, we have 
 \begin{equation*}
  I_t^\alpha \partial_t^\alpha \varphi(t)=
  \begin{cases}
   \varphi(t)-\varphi(0),& \mbox{ if $\alpha\in(0,1)$ and $\varphi\in AC_{\rm loc}([0,\infty))$,}\\
   \varphi(t)-\varphi(0)-t\varphi'(0),&\mbox{ if $\alpha\in(1,2)$ and $\varphi'\in AC_{\rm loc}([0,\infty))$.}
  \end{cases}
 \end{equation*}
 Here $AC_{\rm loc}([0,\infty))$ consists of all the functions belonging to the absolutely continuous function space $AC[0,T]$ for any $T>0$. Also, it is not difficult to check the semigroup property $I_t^{\alpha_1}I_t^{\alpha_2}=I_t^{\alpha_1+\alpha_2}$.
 Hence, recalling \eqref{F}, we can write $I_t^\alpha F$ as 
 \begin{equation*}
  I_t^\alpha F(x,t,\o)=I_t^{\alpha+\delta} f_1(x,t,\o)
  +\Gamma(\alpha+\delta)^{-1}\int_0^t 
  (t-\tau)^{\alpha+\delta-1} f_2(x,\tau,\o)\ dB(\tau).
 \end{equation*}

With the properties of $I_t^\alpha$, we can give the definition of the stochastic weak solution. 
\begin{definition}
\label{def-solution}
We say that $u(x,t,\o):D\times(0,\infty)\times\Omega\to \mathbb{R}$ is a stochastic weak solution of the initial-boundary value problem \eqref{equ-gov} in $L^2(D)$ 
if $u\in C([0,\infty);L^2(D\times\Omega))$, and for a.e. $\omega\in\Omega$, $t\in (0,\infty)$ and all $\varphi\in H^2(D)\cap H_0^1(D)$,
\begin{equation}
\label{equ-var}
\begin{aligned}
&\l u(\cdot,t,\o),\varphi(\cdot) \rd+ I_t^\alpha\l u(\cdot,t,\o),\mathcal A\varphi(\cdot)\rd\\ 
&=\begin{cases}
\l u_0(\cdot,\o)+I_t^\alpha F(\cdot,t,\o),\varphi(\cdot) \rd, &\text{if}\ \alpha\in(0,1),\\
\l u_0(\cdot,\o) + t u_1(\cdot,\o) + I_t^\alpha F(\cdot,t,\o),\varphi(\cdot) \rd, &\text{if}\ \alpha\in (1,2).
\end{cases}
\end{aligned}
\end{equation}
\end{definition}


\begin{remark}
The initial conditions of equation \eqref{equ-gov} are reflected by 
the terms\\ $\l u_0(\cdot,\o),\varphi(\cdot)\rd$ and 
$\l u_0(\cdot,\o) + t u_1(\cdot,\o),\varphi(\cdot)\rd$.
\end{remark}

In Section \ref{sec-pre}, we will introduce the Mittag-Leffler function $E_{\alpha,\beta}(\cdot)$, the Dirichlet eigensystem $\{\lambda_n,\phi_n\}_{n=1}^\infty$ of operator $\A$ and the probability space $(\Omega,\mathcal{F},P)$. Also we let $H_0^1(D)$ and $H^n(D)$ be the Sobolev spaces (e.g., Adams \cite{adams1975sobolev}), and denote the inner product in $L^2(D)$ by $\l\cdot,\cdot\rd$. With the above notations, we can state the main theorems of this work, which concern with the well-posedness of equation \eqref{equ-gov}.    
 \begin{theorem}\label{thm-initial}
  Set $\alpha\in(0,1),$ $T\in(0,\infty),$ $F=0$ and let $u_0(x,\o)$ satisfy Assumption \ref{assumption}. 
  Then under Definition \ref{def-solution}, there exists a unique stochastic weak solution 
  $u\in L^2(D\times(0,\infty)\times\Omega)\cap C([0,\infty);L^2(D\times\Omega))$ for \eqref{equ-gov}.
  The solution $u$ has the form 
\begin{equation}
\label{solution_initial}
 u(x,t,\o)=\sum_{n=1}^\infty \l u_0(\cdot,\o),\phi_n\rd E_{\alpha,1}(-\lambda_n t^\alpha)\phi_n(x), 
\end{equation}
 and satisfies the following regularity properties  
 \begin{equation*}
  \begin{aligned}
  \|u\|_{C([0,\infty);L^2(D\times\Omega))}&\le 
 C\|u_0\|_{L^2(D\times\Omega)},\\
\|u\|_{L^2(D\times(0,T)\times\Omega)}&\le CT^{(1-\alpha)/2}\|u_0\|_{L^2(D\times\Omega)},\\
\|u\|_{L^2((0,T)\times\Omega;H^2(D))}&\le CT^{(1-\alpha)/2}\|u_0\|_{L^2(\Omega;H^1(D))}.
  \end{aligned}
 \end{equation*}
 \end{theorem}
 
The next theorem considers the case of $\alpha\in(1,2)$.
\begin{theorem}
\label{thm-wave-initial}
For $\alpha\in(1,2)$ and $T\in(0,\infty)$, we set $F=0$ and let  
$u_0,u_1$ satisfy Assumption \ref{assumption}. Then under 
Definition \ref{def-solution}, the initial-boundary value problem \eqref{equ-gov} has a 
unique stochastic weak solution $u$, which admits the representation
\begin{equation}
\label{sol-initial-wave}
\begin{aligned}
u(x,t,\o)=&\sum_{n=1}^\infty \l u_0(\cdot,\o),\phi_n\rd E_{\alpha,1}(-\lambda_nt^\alpha) \phi_n(x)
\\
&+ \sum_{n=1}^\infty \l u_1(\cdot,\o),\phi_n\rd tE_{\alpha,2}(-\lambda_nt^\alpha) \phi_n(x).
\end{aligned}
\end{equation}
Moreover, there exists a constant $C>0$ such that
\begin{align*}
\|u\|_{C([0,\infty);L^2(D\times\Omega))}
&\le C \big(\|u_0\|_{L^2(D\times\Omega)} + \|u_1\|_{L^2(D\times\Omega)} \big),\\
\|u\|_{C([0,T];L^2(\Omega; H^1(D)))}
&\le C \big(\|u_0\|_{L^2(\Omega;H^1(D))} + T^{1-\alpha/2}\|u_1\|_{L^2(D\times\Omega)} \big),\\
\|u\|_{L^2((0,T)\times\Omega;H^{2\gamma}(D))}
&\le C\big(T^{(1+\alpha-2\alpha\gamma)/2}\|u_0\|_{L^2(\Omega;H^1(D))} + T^{(3-2\alpha\gamma)/2}\|u_1\|_{L^2(D\times\Omega)}\big),\\
\|u\|_{L^2(\Omega;L^1(0,T;H^2(D)))}
&\le C\big( T^{1-\alpha/2}\|u_0\|_{L^2(\Omega;H^1(D))} + T^{2-\alpha}\|u_1\|_{L^2(D\times\Omega)} \big),
\end{align*}
where $\gamma\in[1/2,\frac{\alpha+1}{2\alpha})$.
\end{theorem}

When $u_0=u_1=0$, we can give the following theorem. 
\begin{theorem}\label{thm-source}
 Let $\alpha\in(0,1)\cup(1,2),$ $\delta$ satisfy the condition \eqref{condition_delta}, $T\in(0,\infty),$ $u_0=u_1=0$, and  $f_1,f_2$ satisfy Assumption \ref{assumption}. 
 Then \eqref{equ-gov} has a unique stochastic weak solution $u$ in the sense of Definition \ref{def-solution}, represented as  
 \begin{equation}\label{equality_4}
 \begin{aligned}
  u(x,t,\o)=&\sum_{n=1}^\infty\Big[\int_0^t \l f_1(\cdot,t-\tau,\omega),\phi_n \rd  \ \tau^{\alpha+\delta-1}E_{\alpha,\alpha+\delta} (-\lambda_n \tau^\alpha)\ d\tau \\
  &\qquad+\int_0^t \l f_2(\cdot,t-\tau,\omega),\phi_n\rd
  \ \tau^{\alpha+\delta-1}E_{\alpha,\alpha+\delta}
  (-\lambda_n \tau^\alpha)\ dB(\tau)\Big]\ \phi_n(x).
  \end{aligned}
 \end{equation}
 
The following regularity estimates of $u$ are valid,  
\begin{equation*}
 \begin{aligned}
  \|u\|_{C([0,\infty);L^2(D\times\Omega))}&\le C\big(\|f_1\|_{L^2(D\times(0,\infty)\times\Omega)}+\|f_2\|_{C([0,\infty);L^2(D\times\Omega))}\big),\\
\|u\|_{L^2(D\times(0,T)\times\Omega)}
 &\le C \big( T^{\alpha+\delta}\|f_1\|_{L^2(D\times(0,T)\times\Omega)}
 +T^{\alpha+\delta-1/2}\|f_2\|_{L^2(D\times(0,T)\times\Omega)}\big),\\
 \|u\|_{L^2((0,T)\times \Omega;H^1(D))}&\le C\big(T^{\alpha/2+\delta}\|f_1\|_{L^2(D\times(0,T)\times\Omega)}+\|f_2\|_{L^2((0,T)\times \Omega;H^1(D))}\big).
 \end{aligned}
\end{equation*}
In addition, when $\alpha\in(1,2)$, we have $u\in C([0,\infty);L^2(\Omega;H^1(D)))$ and  
\begin{equation*}
 \begin{aligned}
  \|u\|_{C([0,T];L^2(\Omega;H^1(D)))}\le CT^{(\alpha-1)/2}
  \big(\|f_1\|_{L^2(D\times(0,T)\times \Omega)}+\|f_2\|_{C([0,\infty);L^2(D\times\Omega))}\big).
 \end{aligned}
\end{equation*}
For the $L^2$-regularity estimates on $t\in[0,\infty)$, provided $$f_1\in L^2(D\times(0,\infty)\times\Omega),\ f_2\in L^2((0,\infty)\times \Omega;H^2(D)),\ \alpha\in(1/2,1),$$ 
we obtain that     
\begin{equation*}
 \begin{aligned}
  \|u\|_{L^2(D\times(0,\infty)\times\Omega)}&\le C \big( \|f_1\|_{L^2(D\times(0,\infty)\times\Omega)}
 +\|f_2\|_{L^2(D\times(0,\infty)\times\Omega)}\big),\\
  \|u\|_{L^2((0,\infty)\times \Omega;H^2(D))}&\le C\big(\|f_1\|_{L^2(D\times(0,\infty)\times\Omega)}
  + \|f_2\|_{L^2((0,\infty)\times \Omega;H^2(D))}\big).
 \end{aligned}
\end{equation*}
\end{theorem}

The remainder of the article is organized as follows. Some preliminary knowledge, such as the fractional calculus, the definitions of stochastic weak solutions and the probability settings are collected in Section \ref{sec-pre}. Then we prove Theorems \ref{thm-initial}, \ref{thm-wave-initial} and \ref{thm-source} in Sections \ref{sec-initial1}, \ref{sec-wave-initial} and \ref{sec-source}, respectively. Section \ref{sec-isp} is devoted to an inverse problem on the determination of the spatially dependent sources in the STFDEs. We prove Theorem \ref{thm-isp}, which is the uniqueness theorem of the inverse problem under the partial boundary measurements. Finally, the concluding remarks are given in Section \ref{sec-con}.

Throughout the paper, the notation $C$ denotes the generic constant which may change at each occurrence.  


\section{Preliminaries.}
\label{sec-pre}

\subsection{Probability space and the Ito formula.} 
\label{sub-ito}
In this work, we will use the knowledge of some statistical moments, like expectation $\E$ and variance $\mathrm{Var}$, which are defined as follows.  
\begin{definition}
 We call $(\Omega,\mathcal{F},P)$ a probability space if $\Omega$ denotes the nonempty sample space, $\mathcal{F}$ is the $\sigma$-algebra of $\Omega$ and $P:\mathcal{F}\to [0,1]$ is the probability measure. Furthermore, we denote $\mathcal{B}$, $\mathcal D$, $\mathcal{F}_t$ as the Borel $\sigma$-algebra on $[0,\infty)$, the $\sigma$-algebra of Lebesgue measurable sets in $D\subset \mathbb{R}^{d}$ and the $\sigma$-algebra generated by the random variables $\{B(s)\}_{0\le s\le t}$, respectively. 
 
 For a random variable $X$ defined on $(\Omega,\mathcal{F},P)$, 
 the expectation $\E$ and variance $\mathrm{Var}$ are defined as follows: 
 \begin{equation*}
  \E[X]=\int_\Omega X(\o)\ dP(\o),\quad 
  \mathrm{Var}[X]=\E\left[(X-\E[X])^2 \right].
 \end{equation*}
\end{definition}

Next we state the Ito isometry formula, which plays an important role in the future proof.  
\begin{lemma}[\cite{Bernt2003stochastic}] \label{Ito} 
Let $(\Omega,\mathcal{F},P)$ be a probability space and let $f,g : [0,\infty)\times \Omega\rightarrow\mathbb{R}$
satisfy the following properties:
\begin{itemize}
 \item[$(1)$] $(t,\omega)\mapsto f(t,\omega)$ is $\mathcal{B}\times\mathcal{F}$-measurable;
 \item[$(2)$] $f(t,\omega)$ is $\mathcal{F}_t$-adapted;
 \item[$(3)$] $\E \left[\int_{T_1}^{T_2}f^2(t,\omega)\ dt \right]<\infty$ for some $0\le T_1\le T_2$.
\end{itemize}
Then it follows that
\begin{equation*}
	\label{eq:ito_isometry}
 \E\Big[\big(\int_{T_1}^{T_2}f(t,\omega)\ dB(t)\big)\ 
 \big(\int_{T_1}^{T_2}g(t,\omega)\ dB(t)\big)\Big]
=\E\Big[\int_{T_1}^{T_2}f(t,\omega)g(t,\omega)\ dt\Big].
\end{equation*}
\end{lemma}

\subsection{Eigensystem of  $\A$ and the Mittag-Leffler function.}
Since $\A$ is a symmetric uniformly elliptic operator, its eigensystem $\{\lambda_n, \phi_n(x)\}_{n=1}^\infty$ has the following properties:   
$0<\lambda_1\le\lambda_2\le\cdots\le\lambda_n \le \cdots$ 
and $\{\phi_n\}_{n=1}^\infty\subset H^2(D)\cap H_0^1(D)$ constitutes an orthonormal basis of $L^2(D).$ Then by Pazy \cite{Pazy1983}, the fractional power $\mathcal A^\gamma,\ \gamma\in\mathbb R$ can be defined as 
$$\mathcal A^\gamma \varphi=\sum_{n=1}^\infty \lambda_n^\gamma \l \varphi,\phi_n\rd\phi_n,\ \varphi\in L^2(D),$$ 
with domain 
$$
D(\mathcal A^\gamma) = \Big\{\varphi\in L^2(D): 
\sum_{n=1}^\infty\left|\lambda_n^\gamma \l \varphi,\phi_n\rd \right|^2<\infty\Big\}.
$$
The domain $D(\mathcal A^\gamma)$ is a Hilbert space with the norm
$$
\|\varphi\|_{D(\mathcal A^\gamma)}=\Big(
\sum_{n=1}^\infty\left|\lambda_n^\gamma \l \varphi,\phi_n\rd \right|^2\Big)^{1/2}.
$$
We note that $D(\mathcal A^\gamma)\subset H^{2\gamma}(D)$ for $\gamma>0$, which gives $\|\cdot\|_{H^{2\gamma}(D)}\le C\|\cdot\|_{D(\mathcal A^\gamma)}$. This inequality will be used frequently in the future proofs. Also we have $D(\mathcal A^{1/2}) = H^1_0(D)$ and $D(\mathcal A) = H^2(D) \cap H^1_0(D)$. The above mentioned results for the Hilbert spaces $D(\mathcal A^\gamma)$, $\gamma>0$ can be found in e.g., \cite{LionsMagenes:1972, LionsMagenes:1972(2), LionsMagenes:1973}.

In the main theorems, we see that the solution $u$ is in terms of the Mittag-Leffler function $E_{\alpha,\beta}(\cdot)$. Next we will give some knowledge of this function. The Mittag-Leffler function is defined as 
\begin{equation*}
E_{\alpha,\beta}(z) = \sum_{k=0}^\infty \frac{z^k}{\Gamma(k\alpha+\beta)},\quad z\in\mathbb C.
\end{equation*}
Notice that this expression generalizes the natural exponential function since 
$E_{1,1}(z)=e^z.$ 

The next lemma concerns the asymptotic property of the Mittag-Leffler function.   
\begin{lemma} \cite[Theorem 1.4]{Podlubny1999fractional}\label{mittag_bound}
Let $0<\alpha<2$ and $\beta\in\mathbb{R}$ be arbitrary. Then it holds that
\begin{equation}
	\label{eq:ML_decay_rate}
  |E_{\alpha,\beta}(-t)| \leq \frac{C}{1+t},\quad t\ge0,
\end{equation}
where the constant $C>0$ is independent of $t>0$. Also for $t\geq 0$ and $p \in {\mathbb N}^+$,  we have the asymptotic formula
\begin{equation*}
	\label{eq:ML_asymptotic_formula}
  E_{\alpha,\beta}(-t)=-\sum_{k=1}^p \frac{(-t)^{-k}}{\Gamma(\beta-\alpha k)}
  +{\mathcal O}(t^{-1-p}),\quad \mbox{as $t \to \infty$.}
\end{equation*}
\end{lemma}

A function $f: (0,\infty) \to \mathbb{R}$ is called completely monotonic if 
$f\in C^\infty(0,\infty)$ and 
\begin{equation*}
	(-1)^n f^{(n)} (t) \geq 0
\end{equation*}
for all $t\in (0,\infty)$, i.e. the derivatives are alternating in sign. The next lemma contains the complete monotonicity of $E_{\alpha,\beta}(\cdot)$. 
\begin{lemma}\label{mittag_monotone}
\par For $0<\alpha<1$ and $\alpha\le \beta$,  
functions $t\mapsto E_{\alpha,\beta}(-t)$ is completely monotonic.
\end{lemma}
\begin{proof}
 See \cite{pollard1948completely} and \cite[Section 4.10.2]{gorenflo2014mittag}.
\end{proof}

At last, we introduce the following two formulas.  
\begin{lemma}{\cite[$(4.4.4)$]{gorenflo2014mittag}}
\label{lem-ml-int}
For $\alpha>0$ and $\beta>0$, 
 \begin{equation*}
  \int_0^t  E_{\alpha,\beta}(\lambda \tau^\alpha) \tau^{\beta-1}\ d\tau
  =t^\beta E_{\alpha,\beta+1}(\lambda t^\alpha).
 \end{equation*}
\end{lemma}

\begin{lemma}{\cite[Lemma $3.2$]{sakamoto2011initial}}\label{mittag_derivative}
	For $\lambda>0,\ \alpha>0$ and $\N+,$ we have that for $t>0$, 
	$$\frac{d^n}{dt^n} \big( E_{\alpha,1}(-\lambda t^\alpha) \big)
	=-\lambda t^{\alpha-n}E_{\alpha,\alpha-n+1}(-\lambda t^\alpha).$$
\end{lemma}

\section{Proof of Theorem \ref{thm-initial}.}\label{sec-initial1}
In this section we set $\alpha\in(0,1)$ and $f_1=f_2=0$. 

\subsection{Continuous regularity. }\label{sec-subdiff-conti}
We prove the continuous regularity of the function $u$ defined in $\eqref{solution_initial}$ first. See the next lemma for details. 
\begin{lemma}\label{lemma_regularity_c_initial}
Let $u_0$ satisfy Assumption \ref{assumption}, then the function $u$ defined in \eqref{solution_initial} belongs to $C([0,\infty);L^2(D\times\Omega))$ and 
 \begin{equation*}
 \|u\|_{C([0,\infty);L^2(D\times\Omega))}\le 
 C\|u_0\|_{L^2(D\times\Omega)}.
 \end{equation*}
\end{lemma}
\begin{proof}
 Fixing $t\ge 0$ and considering $h\in [-t,\infty)$, we have 
 \begin{align*}
 &\|u(\cdot,t+h,\cdot)-u(\cdot,t,\cdot)\|^2_{L^2(D\times\Omega)}\\
 &=\E\Big[\|u(\cdot,t+h,\o)-u(\cdot,t,\o)\|^2_{L^2(D)}\Big]\\
 &=\E\Big[\sum_{n=1}^\infty \l u_0(\cdot,\o),\phi_n\rd^2 \big[E_{\alpha,1}
 (-\lambda_n (t+h)^\alpha) -E_{\alpha,1}(-\lambda_n t^\alpha)\big]^2\Big]\\
 &=\sum_{n=1}^\infty \E\left[\l u_0(\cdot,\o),\phi_n\rd^2 \right]\ \big[E_{\alpha,1}
 (-\lambda_n (t+h)^\alpha) -E_{\alpha,1}(-\lambda_n t^\alpha)\big]^2\\
 &=:\sum_{n=1}^\infty S_n(h).
 \end{align*}
 The monotone convergence theorem is used above. Throughout this work, the monotone convergence theorem will be used frequently. We may omit to mention this in the future proofs.  
 From Lemma \ref{mittag_bound}, we see that for $h\in [-t,\infty)$, 
 \begin{equation*}
  \sum_{n=1}^\infty |S_n(h)|\le C \sum_{n=1}^\infty \E\Big[\l u_0(\cdot,\o),\phi_n\rd^2 \Big]=C \|u_0\|^2_{L^2(D\times \Omega)}<\infty. 
 \end{equation*}
This gives the uniform convergence of the series $\sum_{n=1}^\infty S_n(h)$ 
on $[-t,\infty)$. Also, obviously we can see $S_n(h)$ is continuous on $[-t,\infty)$ for each $\N+$. So, by the uniform limit theorem, we have the continuity of $\sum_{n=1}^\infty S_n(h)$ on $[-t,\infty)$, which leads to 
\begin{equation*}
 \lim_{h\to 0} \|u(\cdot,t+h,\cdot)-u(\cdot,t,\cdot)\|^2_{L^2(D\times\Omega)}= \lim_{h\to 0}\sum_{n=1}^\infty S_n(h)=\sum_{n=1}^\infty S_n(0)=0.
\end{equation*}
This gives $u\in C([0,\infty);L^2(D\times\Omega))$.

Furthermore, we have for $t\in[0,\infty)$, 
\begin{equation*}
\begin{aligned}
 \|u(\cdot,t,\cdot)\|^2_{L^2(D\times\Omega)}
 &=\int_\Omega\sum_{n=1}^\infty\l u_0(\cdot,\o),\phi_n\rd^2
 \ E^2_{\alpha,1}(-\lambda_n t^\alpha)\ dP(\omega)\\
 &\le C\int_\Omega\sum_{n=1}^\infty \l u_0(\cdot,\o),\phi_n\rd^2\ dP(\omega)\\
 &=C\|u_0\|^2_{L^2(D\times\Omega)},
 \end{aligned}
\end{equation*}
which means $ \|u\|_{C([0,\infty);L^2(D\times\Omega))} 
 \le C\|u_0\|_{L^2(D\times\Omega)}$.
This completes the proof of the lemma. 
\end{proof}

\subsection{$L^2$-regularity.}
Next we consider the $L^2$-regularity of the function $u$ defined in \eqref{solution_initial}, which is included by the following lemma.  
\begin{lemma}\label{lemma_regularity_lh_initial}
Under Assumption \ref{assumption}, for $T>0$, the function $u$ defined in \eqref{solution_initial} satisfies 
  \begin{equation*}
   \begin{aligned}
\|u\|_{L^2(D\times(0,T)\times\Omega)}
 &\le CT^{(1-\alpha)/2}\|u_0\|_{L^2(D\times\Omega)},\\
      \|u\|_{L^2((0,T)\times\Omega;H^2(D))}
      &\le CT^{(1-\alpha)/2}\|u_0\|_{L^2(\Omega;H^1(D))}.
  \end{aligned}
 \end{equation*}
\end{lemma}
\begin{proof}
 For the representation \eqref{solution_initial}, we employ Lemma \ref{mittag_bound} to derive that 
 \begin{equation*}
  \begin{aligned}
   \|u(\cdot,t,\omega)\|^2_{L^2(D)}
   &=\sum_{n=1}^\infty \l u_0(\cdot,\o),\phi_n\rd^2 E_{\alpha,1}^2(-\lambda_nt^\alpha)\\
   &\le \sum_{n=1}^\infty \l u_0(\cdot,\o),\phi_n\rd^2 \Big(\frac{C}{1+\lambda_nt^\alpha}\Big)^2\\
   &= C\sum_{n=1}^\infty \l u_0(\cdot,\o),\phi_n\rd^2\lambda_n^{-1}t^{-\alpha}
   \Big(\frac{\lambda_n^{1/2}t^{\alpha/2}}{1+\lambda_nt^\alpha}\Big)^2\\
   &\le C\lambda_1^{-1}t^{-\alpha}\|u_0(\cdot,\omega)\|_{L^2(D)}^2,
  \end{aligned}
 \end{equation*}
which gives
\begin{equation*}
\begin{aligned}
  \|u\|^2_{L^2(D\times(0,T)\times\Omega)}
  =&\int_{(0,T)\times\Omega} \|u(\cdot,t,\omega)\|^2_{L^2(D)}
  \ dt\ dP(\omega)\\
  \le& CT^{1-\alpha}\|u_0\|^2_{L^2(D\times\Omega)}.
\end{aligned} 
\end{equation*}

Now if we further assume $u_0\in L^2(\Omega;H_0^1(D))$, by noting the inequality $\|u\|_{H^2(D)} \le C\|u\|_{D(\mathcal{A})}$ with $u\in H^2(D)\cap H_0^1(D)$, we can analogously obtain
  \begin{align*}
   \|u(\cdot,t,\omega)\|^2_{H^2(D)}
   &\le C\sum_{n=1}^\infty \lambda_n^2\l u_0(\cdot,\o),\phi_n\rd^2 E_{\alpha,1}^2(-\lambda_nt^\alpha)\\
   &\le C\sum_{n=1}^\infty \lambda_n^2\l u_0(\cdot,\o),\phi_n\rd^2 \Big(\frac{C}{1+\lambda_nt^\alpha}\Big)^2\\
   &= C\sum_{n=1}^\infty \l u_0(\cdot,\o),\phi_n\rd^2\lambda_nt^{-\alpha}
   \Big(\frac{\lambda_n^{1/2}t^{\alpha/2}}{1+\lambda_nt^\alpha}\Big)^2\\
   &\le Ct^{-\alpha}\|\A^{1/2}u_0(\cdot,\omega)\|_{L^2(D)}^2
   \le Ct^{-\alpha}\|u_0(\cdot,\omega)\|_{H^1(D)}^2.
  \end{align*}
This gives 
\begin{equation*}
  \|u\|_{L^2((0,T)\times\Omega;H^2(D))}^2\le CT^{1-\alpha}\|u_0\|^2_{L^2(\Omega;H^1(D))},
\end{equation*}
and completes the proof of the lemma. 
\end{proof}

\subsection{Existence and uniqueness.}
At last, the existence and uniqueness of the weak solution are discussed in the next lemma. 
\begin{lemma}\label{lemma_existence_initial}
 Under Definition \ref{def-solution}, $u$ defined by \eqref{solution_initial} is 
 the unique stochastic weak solution of initial-boundary value problem \eqref{equ-gov}.
\end{lemma}
\begin{proof}
 We assume that $u$ is the stochastic weak solution of \eqref{equ-gov} in the sense of Definition \ref{def-solution}. Given $\N+$, \eqref{equ-var} with $\varphi=\phi_n$ implies that for $P$-a.e.  $\omega\in \Omega$, we have
 \begin{equation*}
  \langle u(\cdot,t,\o),\phi_n \rangle_{L^2(D)} 
+ \lambda_nI_t^\alpha\langle   u(\cdot,t,\o),\phi_n\rangle_{L^2(D)} 
= \langle u_0(\cdot,\o),\phi_n \rangle_{L^2(D)},
 \end{equation*}
i.e. 
\begin{equation}\label{ODE_1}
 v_n(t)+\lambda_nI_t^\alpha  v_n(t)= \langle u_0(\cdot,\o),\phi_n \rangle_{L^2(D)},
\end{equation}
where $v_n(t,\o):=\l u(\cdot,t,\o),\phi_n\rd$. 
By \cite[Theorem 3.25]{KilbasSrivastavaTrujillo:2006}, we see that \eqref{ODE_1} is uniquely solvable.  Hence, with the completeness of $\{\phi_n\}_{n=1}^\infty$ in $H^2(D)\cap H_0^1(D)$, 
 we have the uniqueness of the stochastic weak solution to the initial-boundary value problem \eqref{equ-var}. Next we need to deduce the representation of $v_n$ for proving the existence of the solution to \eqref{equ-var}. 
 
We denote  
 \begin{equation*}
  \tilde v_n(t,\o)=\l u_0(\cdot,\o),\phi_n\rd E_{\alpha,1}(-\lambda_n t^\alpha)
 \end{equation*} 
 and see that $\tilde v_n(0,\o)=\langle u_0(\cdot,\o),\phi_n \rd$. 
In addition, we have 
 \begin{equation}\label{e_1}
  \begin{aligned}
   \lambda_nI_t^\alpha  \tilde v_n(t,\o)&=\lambda_n\l u_0(\cdot,\o),\phi_n\rd\I E_{\alpha,1}(-\lambda_n t^\alpha)\\
   &=\l u_0(\cdot,\o),\phi_n\rd[1-E_{\alpha,1}(-\lambda_nt^\alpha)]\\
   &=\l u_0(\cdot,\o),\phi_n\rd-\tilde v_n(t,\o).
  \end{aligned}
 \end{equation}
 We see that $\tilde v_n$ and $v_n$ both satisfy equation \eqref{ODE_1}. Recalling the uniqueness of  the solution of \eqref{ODE_1}, we have $v_n=\tilde v_n$. 
  
 Recalling that $\mathcal A\phi_n=\lambda_n\phi_n$ and $\{\phi_n\}_{n=1}^\infty$ is an orthonormal basis of $L^2(D)$. For any $\varphi\in H^2(D)\cap H_0^1(D)$, we sum up \eqref{e_1} with the coefficients $\langle \varphi, \phi_n \rangle_{L^2(D)}$ for each $n\in\mathbb N^+$. Then we see that $u(x,t,\o)=\sum_{n=1}^\infty \tilde v_n(t,\o)\phi_n(x)$, which is actually the representation \eqref{solution_initial}, satisfies equation \eqref{equ-var}. Also, with the continuity ensured by Lemma 
 \ref{lemma_regularity_c_initial}, we can conclude that $u$ given in \eqref{solution_initial} is the stochastic weak solution of the initial-boundary value problem \eqref{equ-gov} in the sense of Definition \ref{def-solution}. The proof is complete. 
\end{proof}

With the above results, we can deduce Theorem \ref{thm-initial} straightforwardly.
\begin{proof}[Proof of Theorem \ref{thm-initial}]
 Lemma \ref{lemma_existence_initial} ensures the existence and uniqueness 
  of the stochastic weak solution, 
 while Lemmas \ref{lemma_regularity_c_initial} and \ref{lemma_regularity_lh_initial} 
 yield the regularity results.
\end{proof}

\section{Proof of Theorem \ref{thm-wave-initial}.}\label{sec-wave-initial}
In this section, we consider the case of $\alpha\in (1,2)$ and $f_1=f_2=0$. 

\subsection{Continuous regularity.}
We give the estimate of continuous regularity of the function $u$ defined by \eqref{sol-initial-wave} first. 
\begin{lemma}
\label{lem-wave-initial-continuous}
For any fixed $T>0$, we have $u$ in \eqref{sol-initial-wave} belongs to $ C([0,T];L^2(\Omega;H^1(D)))$ and the following regularity estimate holds 
$$
\|u\|_{C([0,T];L^2(\Omega; H^1(D)))}
\le C \Big( \|u_0\|_{L^2(\Omega;H^1(D))} + T^{1-\alpha/2}\|u_1\|_{L^2(D\times\Omega)} \Big).
$$
\end{lemma}
\begin{proof}
Fixing $t\ge 0$ and a sufficiently small $h\in [-t,\infty)$, in view of the inequality $\|u\|_{H^1(D)}\le C\|u\|_{D(\mathcal A^{1/2})}$, we have 
 \begin{align*}
 &\|u(\cdot,t+h,\cdot)-u(\cdot,t,\cdot)\|^2_{L^2(\Omega;H^1(D))}\\
 &=\E\left[\|u(\cdot,t+h,\o)-u(\cdot,t,\o)\|^2_{H^1(D)}\right]\\
 &\le C\sum_{n=1}^\infty \lambda_n \E\Big[\l u_0(\cdot,\o),\phi_n\rd^2 \Big]\ \left[E_{\alpha,1}
 (-\lambda_n (t+h)^\alpha) -E_{\alpha,1}(-\lambda_n t^\alpha)\right]^2\\
&\quad+C\sum_{n=1}^\infty \lambda_n\E\Big[\l u_1(\cdot,\o),\phi_n\rd^2\Big]\ \big[ (t+h)E_{\alpha,2}
 (-\lambda_n (t+h)^\alpha) - t E_{\alpha,2}(-\lambda_n t^\alpha)\big]^2 .
  \end{align*}
Similar to the argument in Section \ref{sec-subdiff-conti},  we first see that
\begin{equation*}
\begin{aligned}
 &\lim_{h\to 0} \sum_{n=1}^\infty \lambda_n\E\Big[\l u_0(\cdot,\o),\phi_n\rd^2\Big]\ \big[E_{\alpha,1}
 (-\lambda_n (t+h)^\alpha) -E_{\alpha,1}(-\lambda_n t^\alpha)\big]^2 \\
&=\sum_{n=1}^\infty \E\Big[\l \mathcal A^{1/2}u_0(\cdot,\o),\phi_n\rd^2\Big]
\ \lim_{h\to 0}\big[E_{\alpha,1} (-\lambda_n (t+h)^\alpha) 
-E_{\alpha,1}(-\lambda_n t^\alpha)\big]^2=0.
\end{aligned}
\end{equation*}
Furthermore, from Lemma \ref{lem-ml-int} with $\beta=1$, we find that
$$(t+h)E_{\alpha,2} (-\lambda_n (t+h)^\alpha) - t E_{\alpha,2}(-\lambda_n t^\alpha) 
= \int_t^{t+h} E_{\alpha,1}(-\lambda_n \tau^\alpha)\ d\tau.$$
From the straightforward computation, we have  
\begin{equation}
\label{esti-useful1}
\frac{t^\beta}{1+t} \le C\ \text{for}\ t>0\ \text{and}\ \beta\in[0,1],
\end{equation}
and
\begin{equation}
\label{esti-useful2}
|a^p - b^p| \le |a-b|^p\ \text{for}\ a,b>0\ \text{and}\ p\in(0,1).
\end{equation}
Then we see that
\begin{align*}
&\big|(t+h)E_{\alpha,2} (-\lambda_n (t+h)^\alpha) - t E_{\alpha,2}(-\lambda_n t^\alpha)  \big|\\
&\le \Big|\int_t^{t+h} \frac{C(\lambda_n \tau^\alpha)^{1/2}}{1+\lambda_n \tau^\alpha} (\lambda_n \tau^\alpha)^{-1/2}\ d\tau \Big|
\le \frac{C}{\sqrt{\lambda_n}} \big| (t+h)^{1-\alpha/2} - t^{1 - \alpha/2} \big|
\le  \frac{C}{\sqrt{\lambda_n}} |h|^{1-\alpha/2}, 
\end{align*}
in view of the estimate \eqref{eq:ML_decay_rate} of the Mittag-Leffler functions. Furthermore, we conclude from the condition $u_1\in L^2(D\times\Omega)$ that the series 
$$
\sum_{n=1}^\infty \lambda_n\E\big[\l u_1(\cdot,\o),\phi_n\rd^2 \big]\ \big[ (t+h)E_{\alpha,2} (-\lambda_n (t+h)^\alpha) - t E_{\alpha,2}(-\lambda_n t^\alpha)\big]^2
$$
is uniformly convergent. Hence
\begin{align*}
&\lim_{h\to0} \sum_{n=1}^\infty \lambda_n\E\big[\l u_1(\cdot,\o),\phi_n\rd^2 \big]\ \big[ (t+h)E_{\alpha,2}
 (-\lambda_n (t+h)^\alpha) - t E_{\alpha,2}(-\lambda_n t^\alpha)\big]^2 \\
&=\sum_{n=1}^\infty \E \big[\l u_1(\cdot,\o),\phi_n\rd^2 \big]\ \lim_{h\to0} \lambda_n\big[ (t+h)E_{\alpha,2}
 (-\lambda_n (t+h)^\alpha) - t E_{\alpha,2}(-\lambda_n t^\alpha)\big]^2 = 0.
\end{align*}
This gives that the solution $u:[0,T]\to L^2(\Omega;H^1(D))$ is continuous. Moreover, from the fact that $D(\A^{1/2})=H_0^1(D)$, we have for $t\in[0,\infty)$, 
\begin{equation*}
\begin{aligned}
 &\|u(\cdot,t,\cdot)\|^2_{L^2(\Omega;H^1(D))} \\
 &\le C\int_\Omega\sum_{n=1}^\infty\l \mathcal A^{1/2}u_0(\cdot,\o),\phi_n\rd^2
 \ E^2_{\alpha,1}(-\lambda_n t^\alpha)\ dP(\omega)
 \\
 &\quad+ C\int_\Omega\sum_{n=1}^\infty \lambda_n\l u_1(\cdot,\o),\phi_n\rd^2
 \ t^2 E^2_{\alpha,2}(-\lambda_n t^\alpha)\ dP(\omega)\\
 &\le C\int_\Omega\sum_{n=1}^\infty \l \mathcal A^{1/2}u_0(\cdot,\o),\phi_n\rd^2 \ dP(\omega)
+ C\int_\Omega\sum_{n=1}^\infty \l u_1(\cdot,\o),\phi_n\rd^2 \frac{t^{2-\alpha} (\lambda_n t^\alpha)}{(1+\lambda_nt^\alpha)^2}\ dP(\omega)\\
 &\le C\|u_0\|^2_{L^2(\Omega; H^1(D))} + Ct^{2-\alpha}\|u_1\|_{L^2(D\times\Omega)}^2.
 \end{aligned}
\end{equation*}
In the last inequality the estimate \eqref{esti-useful1} is used. This yields
$$ 
\|u\|_{C([0,T];L^2(\Omega; H^1(D)))}^2  
 \le C\left(\|u_0\|_{L^2(\Omega; H^1(D))}^2 + T^{2-\alpha}\|u_1\|_{L^2(D\times\Omega)}^2 \right).
$$  
The proof of the lemma is complete.
\end{proof}

In the case of $T=\infty$, by slightly changing the above argument, we can show the next lemma. 
\begin{lemma}
\label{lem-wave-initial-infty}
The solution $u$ defined in \eqref{sol-initial-wave} belongs to $C([0,\infty);L^2(D\times\Omega))$ and
$$
\|u\|_{C([0,\infty);L^2(D\times\Omega))}
\le C \left( \|u_0\|_{L^2(D\times\Omega)} + \|u_1\|_{L^2(D\times\Omega)} \right).
$$
\end{lemma}
\begin{proof}
The direct calculation gives that for $t\in[0,\infty)$,
\begin{equation*}
\begin{aligned}
 &\|u(\cdot,t,\cdot)\|^2_{L^2(D\times\Omega)}\\
 &\le C\int_\Omega\sum_{n=1}^\infty \l u_0(\cdot,\o),\phi_n\rd^2 \ dP(\omega)
+ C\int_\Omega\sum_{n=1}^\infty \l u_1(\cdot,\o),\phi_n\rd^2 \frac{t^2 }{(1+\lambda_nt^\alpha)^2}\ dP(\omega)\\
 &\le C\|u_0\|^2_{L^2(D\times\Omega)} + C\|u_1\|_{L^2(D\times\Omega)}^2,
 \end{aligned}
\end{equation*}
which yields 
$$ 
\|u\|_{C([0,\infty);L^2(D\times\Omega))}^2  
 \le C\left(\|u_0\|_{L^2(D\times\Omega)}^2 + \|u_1\|_{L^2(D\times\Omega)}^2 \right).
$$  
The proof is complete. 
\end{proof}

\subsection{$L^2$-regularity.}
Similar to the argument in the above subsection, we can deduce the $L^2$-regularity of the solution $u$ given by \eqref{sol-initial-wave}.
\begin{lemma}
\label{lem-wave-initial-L2}
Under Assumption \ref{assumption}, $u$ in \eqref{sol-initial-wave} belongs to $L^2((0,T)\times\Omega;H^{2\gamma}(D))$ with $\gamma\in [1/2,\frac{\alpha+1}{2\alpha})$ and satisfies 
$$
\|u\|_{L^2((0,T)\times\Omega;H^{2\gamma}(D))}
\le C\Big(T^{1+\alpha-2\alpha\gamma}\|u_0\|_{L^2(\Omega;H^1(D))}^2 + T^{3-2\alpha\gamma}\|u_1\|_{L^2(D\times\Omega)}^2\Big).
$$
\end{lemma}
\begin{proof}
Multiplying $\A^{\gamma}$ on both sides of \eqref{sol-initial-wave}, from the estimate \eqref{eq:ML_decay_rate} of the Mittag-Leffler function, we arrive at
$$
\|u(\cdot,t,\o)\|_{D(\A^\gamma)}^2 
\le C\sum_{n=1}^\infty \Big(\frac{\lambda_n^\gamma\l u_0(\cdot,\o),\phi_n\rd }{1+\lambda_n t^\alpha}\Big)^2  
+ Ct^2\sum_{n=1}^\infty \Big(\frac{\lambda_n^\gamma\l u_1(\cdot,\o),\phi_n\rd }{1+\lambda_n t^\alpha}\Big)^2.
$$
By \eqref{esti-useful1},  the choice of $\gamma\in [1/2,\frac{\alpha+1}{2\alpha})$ and the proof of the above theorem, we can deduce that
\begin{align*}
\|u(\cdot,t,\o)\|_{D(\A^\gamma)}^2 
&\le Ct^{\alpha-2\alpha\gamma} \sum_{n=1}^\infty \lambda_n \l u_0(\cdot,\o),\phi_n\rd ^2\Big(\frac{(\lambda_nt^\alpha)^{\gamma-1/2}}{1+\lambda_n t^\alpha}\Big)^2  
\\
&\quad+ Ct^{2-2\alpha\gamma} \sum_{n=1}^\infty \l u_1(\cdot,\o),\phi_n\rd^2 \Big(\frac{(\lambda_nt^\alpha)^{\gamma} }{1+\lambda_n t^\alpha}\Big)^2
\\
&\le Ct^{\alpha-2\alpha\gamma}\|u_0\|_{D(\A^{1/2})}^2
+  Ct^{2-2\alpha\gamma}\|u_1\|_{L^2(D)}^2.
\end{align*}
Then it follows that
$$
\int_0^T \|u(\cdot,t,\o)\|_{D(\A^\gamma)}^2\ dt
\le C\Big(\frac{T^{1+\alpha-2\alpha\gamma}}{1+\alpha-2\alpha\gamma}\|u_0\|^2_{H^1(D)}
+ \frac{T^{3-2\alpha\gamma}}{3-2\alpha\gamma}\|u_1\|_{L^2(D)}^2\Big).
$$
By virtue of the relation $D(\mathcal A^{\gamma}) \subset H^{2\gamma}(D)$, we finally obtain
$$
\|u\|_{L^2((0,T)\times\Omega;H^{2\gamma}(D))}^2
\le C\Big(T^{1+\alpha-2\alpha\gamma}\|u_0\|^2_{L^2(\Omega;H^1(D))} + T^{3-2\alpha\gamma}\|u_1\|_{L^2(D\times\Omega)}^2\Big).
$$
This completes the proof of the lemma.
\end{proof}

\subsection{$L^1$-regularity.}
Unfortunately, using the above argument in Lemma \ref{lem-wave-initial-L2}, we cannot show the $L^2(0,T;H^2(D))$ 
regularity of the solution $u$ in \eqref{sol-initial-wave} only provided that $u_0\in L^2(\Omega;H^1(D))$ and $u_1\in L^2(D\times\Omega)$. This is due to the lack of the monotonic property of the Mittag-Leffler functions $E_{\alpha,\beta}$ in the case of $\alpha>1$.  
However, with these conditions we can deduce $u(\cdot,\cdot,\omega)\in L^1(0,T;H^2(D))$ for a.e.   $\o\in\Omega$.  
\begin{lemma}
\label{lem-wave-initial-L1}
Let $u_0,u_1$ satisfy Assumption \ref{assumption}, then the function $u$ in \eqref{sol-initial-wave} belongs to $L^2(\Omega;L^1(0,T;H^2(D)))$ and it holds that 
$$\|u\|_{L^2(\Omega; L^1(0,T;H^2(D)))} 
\le C\Big( T^{1-\alpha/2}\|u_0\|_{L^2(\Omega;H^1(D))} + T^{2-\alpha}\|u_1\|_{L^2(D\times\Omega)} \Big).$$
\end{lemma}
\begin{proof}
Multiplying $\A$ on both sides of \eqref{sol-initial-wave} derives that 
\begin{align*}
&\|u(\cdot,t,\o)\|_{D(\A)}^2\\
&= \sum_{n=1}^\infty \lambda_n^2\ \Big|\l u_0(\cdot,\o),\phi_n\rd  E_{\alpha,1}(-\lambda_nt^\alpha) 
+ \l u_1(\cdot,\o),\phi_n\rd  tE_{\alpha,2}(-\lambda_nt^\alpha) \Big|^2\\
&\le C\sum_{n=1}^\infty \l u_0(\cdot,\o),\phi_n\rd ^2 \Big( \frac{\lambda_n}{1+\lambda_nt^\alpha}\Big)^2 + Ct^2 \sum_{n=1}^\infty \l u_1(\cdot,\o),\phi_n\rd ^2 \Big( \frac{\lambda_n}{1+\lambda_nt^\alpha}\Big)^2\\
&=Ct^{-\alpha}\sum_{n=1}^\infty \lambda_n \l u_0(\cdot,\o),\phi_n \rd ^2 \frac{\lambda_nt^\alpha}{(1+\lambda_nt^\alpha)^2 }+ Ct^{2-2\alpha} \sum_{n=1}^\infty \l u_1(\cdot,\o),\phi_n \rd ^2 \Big( \frac{\lambda_nt^\alpha}{1+\lambda_nt^\alpha}\Big)^2\\
&\le C\|u_0(\cdot,\o)\|^2_{D(\A^{1/2})} t^{-\alpha} + C \|u_1(\cdot,\o)\|_{L^2(D)}^2t^{2-2\alpha},
\end{align*}
which gives 
\begin{equation*}
\|u(\cdot,\cdot,\o)\|_{L^1(0,T;H^2(D))}  \le C\|u_0(\cdot,\o)\|_{D(\A^{1/2})} T^{1-\alpha/2} + C \|u_1(\cdot,\o)\|_{L^2(D)}T^{2-\alpha}.
\end{equation*}
Hence,  we have 
$$
\|u\|_{L^2(\Omega; L^1(0,T;H^2(D)))}^2 
\le C\Big( T^{2-\alpha}\|u_0\|_{L^2(\Omega;H^1(D))}^2 + T^{4-2\alpha}\|u_1\|_{L^2(D\times\Omega)}^2 \Big), 
$$
in view of the fact that $D(\mathcal A^{1/2})=H_0^1(D)$. This completes the proof. 
\end{proof}

\subsection{Existence and uniqueness.} 
Finally we show the existence and uniqueness of solution \eqref{sol-initial-wave}. 
\begin{lemma}
\label{lem-wave-initial-existence}
 Under Definition \ref{def-solution}, the function $u$ in \eqref{sol-initial-wave} is the unique stochastic weak solution of equation \eqref{equ-gov}.
\end{lemma}
\begin{proof} 
From Definition \ref{def-solution}, for $\N+$ and a.e. $\o\in \Omega$, we can derive the equation for  $v_n(t,\o)=\l u(\cdot,t,\o),\phi_n(\cdot)\rd$ as  
\begin{equation}\label{ODE_3}
 v_n(t,\o)+\lambda_nI_t^\alpha  v_n(t,\o)= \langle u_0(\cdot,\o)+tu_1(\cdot,\o),\phi_n(\cdot) \rangle_{L^2(D)}.
\end{equation}
The uniqueness of the solution of \eqref{ODE_3} is ensured by \cite[Theorem 3.25]{KilbasSrivastavaTrujillo:2006}. This with the completeness of $\{\phi_n\}_{n=1}^\infty$ in 
$H^2(D)\cap H_0^1(D)$ yields the uniqueness of the stochastic weak solution of the problem \eqref{equ-gov}. 

Define 
$$
\tilde v_n(t,\o) = \l u_0(\cdot,\o),\phi_n(\cdot)\rd  E_{\alpha,1}(-\lambda_nt^\alpha)
+ \l u_1(\cdot,\o),\phi_n(\cdot)\rd  tE_{\alpha,2}(-\lambda_nt^\alpha),
$$
and obviously $\tilde v_n(0,\o)=\l u_0(\cdot,\o),\phi_n(\cdot)\rd$. 
Also, from Lemma \ref{lem-ml-int}, we see that
\begin{equation*}
 \begin{aligned}
  I_t^\alpha (E_{\alpha,1}(-\lambda_nt^\alpha)) &= \lambda_n^{-1} (1 - E_{\alpha,1}(-\lambda_nt^\alpha)),\\
  I_t^\alpha (tE_{\alpha,2}(-\lambda_nt^\alpha)) &= t\lambda_n^{-1} (1 - E_{\alpha,2}(-\lambda_nt^\alpha)).
 \end{aligned}
\end{equation*}
Therefore it holds that 
\begin{equation}\label{e_2}
\tilde v_n(t,\o) + \lambda_n I_t^\alpha \tilde v_n(t,\o) = \l u_0(\cdot,\o)+tu_1(\cdot,\o),\phi_n (\cdot)\rd.
\end{equation}
We recall that $\mathcal A\phi_n=\lambda_n\phi_n$ and $\{\phi_n\}_{n=1}^\infty$ is an orthonormal basis in $L^2(D)$. For any $\varphi\in H^2(D)\cap H_0^1(D)$, we can sum \eqref{e_2} for each $n\in\mathbb N^+$ with coefficient $\langle \varphi,\phi_n \rangle_{L^2(D)}$. Then we have $u$ given in \eqref{sol-initial-wave} satisfies equation \eqref{equ-var}. With the continuity from Lemma \ref{lem-wave-initial-infty}, we can achieve the desired conclusion and complete the proof. 
\end{proof}

Now it is the time to complete the proof of Theorem \ref{thm-wave-initial}. 
\begin{proof}[Proof of Theorem \ref{thm-wave-initial}]
 The existence and uniqueness of the stochastic weak solution follow from Lemma \ref{lem-wave-initial-existence}, and 
 the regularity estimates are given by Lemmas \ref{lem-wave-initial-continuous}, \ref{lem-wave-initial-infty}, \ref{lem-wave-initial-L2} and \ref{lem-wave-initial-L1}.
\end{proof}

\section{Proof of Theorem \ref{thm-source}.}\label{sec-source}
 In this section, we set $u_0(x,\o)=u_1(x,\o)=0$. To shorten the proof, we define the stochastic processes $v_{n,1}(t,\omega)$ and $v_{n,2}(t,\omega)$ with $t>0$ and $\o\in\Omega$ as     
\begin{equation*}
\begin{aligned}
v_{n,1}(t,\omega)&=\int_0^t \l f_1(\cdot,t-\tau,\omega),\phi_n\rd
  \ \tau^{\alpha+\delta-1}E_{\alpha,\alpha+\delta} (-\lambda_n \tau^\alpha)~d\tau,\\
v_{n,2}(t,\omega)&=\int_0^t \l f_2(\cdot,t-\tau,\omega),\phi_n\rd
\ \tau^{\alpha+\delta-1}E_{\alpha,\alpha+\delta}
  (-\lambda_n \tau^\alpha)~dB(\tau), \quad \N+.
\end{aligned}
\end{equation*}

\begin{remark}
 For each $\N+$, $v_{n,2}$ is well-defined in the sense of Lemma \ref{Ito}. To prove this  statement, it is sufficient to show that for $t\in(0,\infty)$,  
 \begin{equation*}
 \E\Big[\int_0^t \l f_2(\cdot,t-\tau,\omega),\phi_n\rd^2
\ \tau^{2\alpha+2\delta-2}E^2_{\alpha,\alpha+\delta}
  (-\lambda_n \tau^\alpha)\ d\tau\Big]<\infty. 
 \end{equation*}
From Lemma \ref{mittag_bound} and $\alpha+\delta>1/2$, it follows that for $t>0$, 
\begin{equation*}
\begin{aligned}
  &\E\Big[\int_0^t \l f_2(\cdot,t-\tau,\omega),\phi_n\rd^2
\ \tau^{2\alpha+2\delta-2}E^2_{\alpha,\alpha+\delta}
  (-\lambda_n \tau^\alpha)\ d\tau\Big]\\
  &\le  C\E\Big[\int_0^t \|f_2(\cdot,t-\tau,\omega)\|^2_{L^2(D)}
  \ \tau^{2\alpha+2\delta-2}\ d\tau\Big]\\
  &\le C \|f_2\|^2_{C([0,\infty);L^2(D\times\Omega))}\ 
  \int_0^t \tau^{2\alpha+2\delta-2}\ d\tau\\
  &\le  C t^{2\alpha+2\delta-1} \|f_2\|^2_{C([0,\infty);L^2(D\times\Omega))}<\infty.
\end{aligned}
\end{equation*}
\end{remark}


\subsection{Continuous regularity.}
Similar to the above sections, we first prove the continuous regularity of $u$ in \eqref{equality_4}.  
\begin{lemma}\label{lemma_continuity}
We assume $\alpha\in(0,1)\cup(1,2)$ and let $f_1,f_2$ satisfy Assumption \ref{assumption}. Then the function $u$ in \eqref{equality_4} belongs to $C([0,\infty);L^2(D\times\Omega)).$ 
\end{lemma}
\begin{proof}
 Given $t\in [0,\infty)$, $\o\in\Omega$ and $h>0$, we have 
  \begin{align*}
   &\l u(\cdot,t+h,\o)-u(\cdot,t,\o) ,\phi_n\rd\\
   &=\int_0^t \l f_1(\cdot,t+h-\tau,\omega)-f_1(\cdot,t-\tau,\omega),\phi_n\rd
   \ \tau^{\alpha+\delta-1}E_{\alpha,\alpha+\delta}(-\lambda_n \tau^\alpha)~d\tau\\
   &\quad+\int_t^{t+h} \l f_1(\cdot,t+h-\tau,\omega),\phi_n\rd
   \ \tau^{\alpha+\delta-1}E_{\alpha,\alpha+\delta}(-\lambda_n \tau^\alpha)~d\tau\\
   &\quad+\int_0^t \l f_2(\cdot,t+h-\tau,\omega)-f_2(\cdot,t-\tau,\omega),\phi_n\rd
   \ \tau^{\alpha+\delta-1}E_{\alpha,\alpha+\delta}(-\lambda_n \tau^\alpha)~dB(\tau)\\
   &\quad+\int_t^{t+h} \l f_2(\cdot,t+h-\tau,\omega),\phi_n\rd
   \ \tau^{\alpha+\delta-1}E_{\alpha,\alpha+\delta}(-\lambda_n \tau^\alpha)~dB(\tau).
  \end{align*}
Then we see that   
\begin{equation*}
 \begin{aligned}
  &\| u(\cdot,t+h,\o)-u(\cdot,t,\o)\|^2_{L^2(D)}\\
  &\le C\sum_{n=1}^\infty \Big[\int_0^t \l f_1(\cdot,t+h-\tau,\omega)-f_1(\cdot,t-\tau,\omega),\phi_n\rd\ \tau^{\alpha+\delta-1}
  E_{\alpha,\alpha+\delta}(-\lambda_n \tau^\alpha)~d\tau\Big]^2\\
   &\quad+C\sum_{n=1}^\infty \Big[\int_t^{t+h} \l f_1(\cdot,t+h-\tau,\omega),\phi_n\rd\ \tau^{\alpha+\delta-1}
   E_{\alpha,\alpha+\delta}(-\lambda_n \tau^\alpha)~d\tau\Big]^2\\
   &\quad+C\sum_{n=1}^\infty \Big[\int_0^t \l f_2(\cdot,t+h-\tau,\omega)-f_2(\cdot,t-\tau,\omega),\phi_n\rd
   \ \tau^{\alpha+\delta-1}E_{\alpha,\alpha+\delta}(-\lambda_n \tau^\alpha)~dB(\tau)\Big]^2\\
   &\quad+C\sum_{n=1}^\infty \Big[\int_t^{t+h} \l f_2(\cdot,t+h-\tau,\omega),\phi_n\rd
   \ \tau^{\alpha+\delta-1}E_{\alpha,\alpha+\delta}(-\lambda_n \tau^\alpha)~dB(\tau)\Big]^2\\
   &=:C(S_1+S_2+S_3+S_4).
 \end{aligned}
\end{equation*}

By the fact that $\alpha+\delta>1/2$ and Lemma \ref{mittag_bound}, it follows that 
\begin{equation}\label{esti-int_ml}
\begin{aligned}
 \int_0^t \tau^{2\alpha+2\delta-2}E^2_{\alpha,\alpha+\delta}(-\lambda_n \tau^\alpha)\ d\tau
 &\le  C\int_0^t \tau^{2\alpha+2\delta-2}\ d\tau=Ct^{2\alpha+2\delta-1},\\
 \int_t^{t+h} \tau^{2\alpha+2\delta-2}E^2_{\alpha,\alpha+\delta}(-\lambda_n \tau^\alpha)\ d\tau
 &\le C\big[(t+h)^{2\alpha+2\delta-1}-t^{2\alpha+2\delta-1} \big].
 \end{aligned}
\end{equation}
To estimate $S_1$, the above inequalities \eqref{esti-int_ml} combined with the H\"older inequality and the monotone convergence theorem derive that
\begin{equation*}
 \begin{aligned}
  \E[S_1]
  \le& Ct^{2\alpha + 2\delta -1}\E\Big[\sum_{n=1}^\infty\int_0^t \l f_1(\cdot,t+h-\tau,\omega)
  -f_1(\cdot,t-\tau,\omega),\phi_n\rd^2\ d\tau
  \ \Big]\\
  \le& Ct^{2\alpha+2\delta-1} \int_0^t \E\Big[\| f_1(\cdot,t+h-\tau,\omega)-f_1(\cdot,t-\tau,\omega)\|^2_{L^2(D)} \Big]\ d\tau.
 \end{aligned}
\end{equation*}
The conditions $f_1\in C([0,\infty);L^2(D\times\Omega))$ and 
$\|f_1\|_{C([0,\infty);L^2(D\times\Omega))}<\infty$ allow us to apply the dominated convergence 
theorem to the above integral. This yields that 
\begin{equation*}
\begin{aligned}
 &\lim_{h\to 0+}\int_0^t \| f_1(\cdot,t+h-\tau,\cdot)-f_1(\cdot,t-\tau,\cdot)\|^2_{L^2(D\times\Omega)}\ d\tau\\
&=\int_0^t \lim_{h\to 0+}\| f_1(\cdot,t+h-\tau,\cdot)-f_1(\cdot,t-\tau,\cdot)\|^2_{L^2(D\times\Omega)}\ d\tau=0.
\end{aligned}
\end{equation*}
Now we have $\E[S_1]\to 0$ as $h\to 0+$.

For $S_2$, using the H\"older inequality and Lemma \ref{mittag_bound} again, we derive from \eqref{esti-int_ml} that
 \begin{align*}
  &\E[S_2]\\
  &\le\E\Big[\sum_{n=1}^\infty \big(\int_t^{t+h} \l f_1(\cdot,t+h-\tau,\omega),\phi_n\rd^2\ d\tau\big)
  \ \big(\int_t^{t+h}\tau^{2\alpha+2\delta-2}E^2_{\alpha,\alpha+\delta}(-\lambda_n \tau^\alpha)~d\tau\big)\Big]\\
  &\le C\big[(t+h)^{2\alpha+2\delta-1}-t^{2\alpha+2\delta-1} \big]\ \int_t^{t+h} \E\Big[\| f_1(\cdot,t+h-\tau,\omega)\|_{L^2(D)}^2 \Big]\ d\tau\\
  &\le C\big[(t+h)^{2\alpha+2\delta-1}-t^{2\alpha+2\delta-1} \big]\ h \|f_1\|^2_{C([0,\infty);L^2(D\times\Omega))}\to 0,\ \text{as}\ h\to 0\!+.
 \end{align*}

For $S_3$, with the Ito formula in Lemma \ref{Ito} and the estimates for Mittag-Leffler functions in Lemma \ref{mittag_bound}, we deduce that
\begin{equation*}
\begin{aligned}
 &\E[S_3]\\
 &=\sum_{n=1}^\infty \E\Big[\big[\int_0^t \l f_2(\cdot,t+h-\tau,\omega)-f_2(\cdot,t-\tau,\omega),\phi_n\rd
   \ \tau^{\alpha+\delta-1}E_{\alpha,\alpha+\delta}(-\lambda_n \tau^\alpha)~dB(\tau)\big]^2\Big]
   \\
 &= \E\Big[\sum_{n=1}^\infty \int_0^t \l f_2(\cdot,t+h-\tau,\omega)-f_2(\cdot,t-\tau,\omega),\phi_n\rd^2
   \ \tau^{2\alpha+2\delta-2}E^2_{\alpha,\alpha+\delta}(-\lambda_n \tau^\alpha)~d\tau \Big]\\
 &\le C \int_0^t \| f_2(\cdot,t+h-\tau,\cdot)-f_2(\cdot,t-\tau,\cdot)\|_{L^2(D\times\Omega)}^2
   \ \tau^{2\alpha+2\delta-2}~d\tau.
\end{aligned}
\end{equation*}
With the facts that $\alpha+\delta>1/2$ and 
$\|f_2\|_{C([0,\infty);L^2(D\times\Omega))}<\infty$, we can apply the dominated convergence  theorem to obtain
\begin{equation*}
\begin{aligned}
 &\lim_{h\to0+}\int_0^t \| f_2(\cdot,t+h-\tau,\cdot)
  -f_2(\cdot,t-\tau,\cdot)\|_{L^2(D\times\Omega)}^2
   \ \tau^{2\alpha+2\delta-2}~d\tau\\
  &= \int_0^t \lim_{h\to0+}\| f_2(\cdot,t+h-\tau,\cdot)
  -f_2(\cdot,t-\tau,\cdot)\|_{L^2(D\times\Omega)}^2
   \ \tau^{2\alpha+2\delta-2}~d\tau=0,
\end{aligned}
\end{equation*} 
which gives $\lim_{h\to0+} \E[S_3]=0$.

For $S_4$,  it follows analogously that
 \begin{align*}
 \E[S_4]&= \sum_{n=1}^\infty \E\Big[\big[\int_t^{t+h} \l f_2(\cdot,t+h-\tau,\omega),\phi_n\rd
   \ \tau^{\alpha+\delta-1}E_{\alpha,\alpha+\delta}(-\lambda_n \tau^\alpha)~dB(\tau)\big]^2\Big]\\
   &=\sum_{n=1}^\infty \E\Big[\int_t^{t+h} \l f_2(\cdot,t+h-\tau,\omega),\phi_n\rd^2
   \ \tau^{2\alpha+2\delta-2}E^2_{\alpha,\alpha+\delta}(-\lambda_n \tau^\alpha)~d\tau \Big]\\
   &\le C\int_t^{t+h} \|f_2(\cdot,t+h-\tau,\cdot)\|^2_{L^2(D\times\Omega)}\ \tau^{2\alpha+2\delta-2}\  d\tau\\
   &\le C\big[(t+h)^{2\alpha+2\delta-1}-t^{2\alpha+2\delta-1} \big]\ \|f_2\|^2_{C([0,\infty);L^2(D\times\Omega))} 
   \to 0,\ \text{as}\ h\to 0\!+.
 \end{align*}
Combining the estimates for $S_j,\ j=1,2,3,4$, we have   
\begin{equation*}
 \lim_{h\to0+}\E\left[\| u(\cdot,t+h,\o)-u(\cdot,t,\o)\|^2_{L^2(D)}\right]=
 \lim_{h\to0+} \| u(\cdot,t+h,\cdot)-u(\cdot,t,\cdot)\|^2_{L^2(D\times\Omega)}=0.
\end{equation*}
The result for the case of $h\to 0-$ can be proved analogously. Hence, we have 
\begin{equation*}
 \lim_{h\to0} \| u(\cdot,t+h,\cdot)-u(\cdot,t,\cdot)\|^2_{L^2(D\times\Omega)}=0
\end{equation*}
and complete the proof of the lemma.
\end{proof}

Next we give the estimate of the continuous regularity.
\begin{lemma}\label{lemma_regularity_c_source}
Under the same assumptions in Lemma \ref{lemma_continuity}, the continuous regularity of the function $u$ in \eqref{equality_4} is given as 
 \begin{equation*}
 \|u\|_{C([0,\infty);L^2(D\times\Omega))}
\le C\big(\|f_1\|_{L^2(D\times(0,\infty)\times\Omega)}
+\|f_2\|_{C([0,\infty);L^2(D\times\Omega))}\big).
 \end{equation*}
\end{lemma}
\begin{proof}
 Given $t\in [0,\infty),$ we can deduce that  
 \begin{equation*}
 \begin{aligned}
 \|u(\cdot,t,\cdot)\|^2_{L^2(D\times\Omega)}
  \le& C \sum_{n=1}^\infty \E\Big[\big(\int_0^t\tau^{\alpha+\delta-1}
  E_{\alpha,\alpha+\delta}(-\lambda_n\tau^\alpha)\ \l f_1(\cdot,t-\tau,\omega),\phi_n\rd ~d\tau \big)^2\Big]\\
  &+C \sum_{n=1}^\infty \E\Big[\big(\int_0^t\tau^{\alpha+\delta-1}
  E_{\alpha,\alpha+\delta}(-\lambda_n\tau^\alpha)\ \l f_2(\cdot,t-\tau,\omega),\phi_n\rd ~dB(\tau) \big)^2\Big].
  \end{aligned}  
 \end{equation*}
The Ito formula in Lemma \ref{Ito} gives that 
\begin{equation*}
\begin{aligned}
 &\E\Big[\big(\int_0^t\tau^{\alpha+\delta-1}
  E_{\alpha,\alpha+\delta}(-\lambda_n\tau^\alpha)\ \l f_2(\cdot,t-\tau,\omega),\phi_n\rd ~dB(\tau) \big)^2\Big]\\
  &=\E\Big[\int_0^t\tau^{2\alpha+2\delta-2}
  E^2_{\alpha,\alpha+\delta}(-\lambda_n\tau^\alpha)\ \l f_2(\cdot,t-\tau,\omega),\phi_n\rd^2 ~d\tau\Big],
  \end{aligned}
\end{equation*}
which combined with the H\"older inequality and the Fubini lemma further implies
\begin{equation*}
 \begin{aligned}
 &\|u(\cdot,t,\cdot)\|^2_{L^2(D\times\Omega)}\\
   &\le  C \sum_{n=1}^\infty  \Big[\int_0^t\tau^{2\alpha+2\delta-2}E^2_{\alpha,\alpha+\delta}
  (-\lambda_n\tau^\alpha)~d\tau\Big]~
  \E\Big[\int_0^t \l f_1(\cdot,t-\tau,\omega),\phi_n\rd^2\ d\tau\Big]\\
  &\quad +C \sum_{n=1}^\infty \int_0^t\tau^{2\alpha+2\delta-2} E^2_{\alpha,\alpha+\delta}(-\lambda_n\tau^\alpha) 
  \ \E\left[\l f_2(\cdot,t-\tau,\omega),\phi_n\rd^2\right]~d\tau\\
  &=:C(S_1+S_2).
  \end{aligned}  
 \end{equation*}
Moreover, Lemma \ref{mittag_bound}, condition \eqref{condition_delta} and the fact $\lambda_1\le \lambda_n$ ensure that   
\begin{equation}\label{inequality_1}
 \begin{aligned}
  \int_0^\infty\tau^{2\alpha+2\delta-2}E^2_{\alpha,\alpha+\delta}(-\lambda_n\tau^\alpha)~d\tau
  &= \Big(\int_0^1+\int_1^\infty\Big)\tau^{2\alpha+2\delta-2}E^2_{\alpha,\alpha+\delta}(-\lambda_n\tau^\alpha)~d\tau \\
  &\le C\Big(\int_0^1+\int_1^\infty\Big)\tau^{2\alpha+2\delta-2} (1+\lambda_n\tau^\alpha)^{-2}~d\tau \\
  &\le C\int_0^1\tau^{2\alpha+2\delta-2} ~d\tau+C\lambda_1^{-2}\int_1^\infty\tau^{2\alpha+2\delta-2}\tau^{-2\alpha}~d\tau\\
  &=\tilde C<\infty.
 \end{aligned}
\end{equation}
 Here $\tilde C>0$ is a constant which is independent of $n$. So we can further estimate $S_1$ as  
\begin{equation*}
\begin{aligned}
S_1&\le C \sum_{n=1}^\infty  \Big[\int_0^\infty\tau^{2\alpha+2\delta-2}E^2_{\alpha,\alpha+\delta} (-\lambda_n\tau^\alpha)~d\tau\Big]~
  \E\Big[\int_0^\infty \l f_1(\cdot,\tau,\omega),\phi_n\rd^2\ d\tau\Big]\\
&\le C \E\Big[\int_0^\infty \sum_{n=1}^\infty \l f_1(\cdot,\tau,\omega),\phi_n\rd^2\ d\tau\Big] =C\|f_1\|^2_{L^2(D\times(0,\infty)\times\Omega)}.  
\end{aligned}
\end{equation*}
For $S_2$, following the technique in \eqref{inequality_1} and the result 
$$ \E\Big[\sum_{n=1}^\infty \l f_2(\cdot,\tau,\omega),\phi_n\rd^2\Big]
 \le\|f_2\|^2_{C([0,\infty);L^2(D\times\Omega))}\ \text{for}\ \tau\ge0,$$  we have 
\begin{equation*}
 \begin{aligned}
  S_2=&C  \int_0^t (\chi_{{}_{\tau\le 1}}+\chi_{{}_{\tau\ge 1}})\sum_{n=1}^\infty\tau^{2\alpha+2\delta-2} E^2_{\alpha,\alpha+\delta}(-\lambda_n\tau^\alpha) \ \E\Big[\l f_2(\cdot,t-\tau,\omega),\phi_n\rd^2\Big]~d\tau\\
  \le& C \|f_2\|^2_{C([0,\infty);L^2(D\times\Omega))}
  \Big( \int_0^1\tau^{2\alpha+2\delta-2} ~d\tau
  + \lambda_1^{-2} \int_1^\infty \tau^{2\delta-2}\ d\tau\Big)\\
  \le& C \|f_2\|^2_{C([0,\infty);L^2(D\times\Omega))}. 
 \end{aligned}
\end{equation*}

Now we can conclude that for $t\in[0,\infty)$, 
\begin{equation*}
  \|u(\cdot,t,\cdot)\|_{L^2(D\times\Omega)}
\le C\big(\|f_1\|_{L^2(D\times(0,\infty)\times\Omega)}
+\|f_2\|_{C([0,\infty);L^2(D\times\Omega))}\big),
\end{equation*}
which gives 
\begin{equation*}
 \|u\|_{C([0,\infty);L^2(D\times\Omega))}
\le C\big(\|f_1\|_{L^2(D\times(0,\infty)\times\Omega)}
+\|f_2\|_{C([0,\infty);L^2(D\times\Omega))}\big)
\end{equation*}
and completes the proof of the lemma. 
\end{proof}

In the case of $\alpha\in(1,2)$,  we can deduce a stronger time-continuity of $u$, which is $u\in C([0,\infty);L^2(\Omega;H^1(D)))$.  
\begin{corollary}
\label{cor-wave-source-conti}
 We let $f_1,f_2$ satisfy Assumption \ref{assumption} and $\alpha\in(1,2)$. Then the function $u$ in \eqref{equality_4} belongs to $C([0,\infty);L^2(\Omega;H^1(D)))$. Also, for each $T\in(0,\infty)$, it holds that 
 \begin{equation*}
  \|u\|_{C([0,T];L^2(\Omega;H^1(D)))}\le CT^{(\alpha-1)/2}
  \big(\|f_1\|_{L^2(D\times(0,T)\times \Omega)}+\|f_2\|_{C([0,\infty);L^2(D\times\Omega))}\big).
  \end{equation*}
\end{corollary}
\begin{proof}
Firstly, we point out that we can set $\delta = 0$ in view of the condition \eqref{condition_delta} and $\alpha\in(1,2)$. Noting the relation $D(\mathcal A^{1/2}) = H_0^1(D)$ and following the proof of Lemma \ref{lemma_continuity}, we have that given $t\in [0,T]\subset[0,\infty)$, $\o\in\Omega$ and $h>0$, 
 \begin{align*}
  &\| u(\cdot,t+h,\o)-u(\cdot,t,\o)\|^2_{H^1(D)}\\
  &\le C\sum_{n=1}^\infty \Big[\int_0^t \l f_1(\cdot,t+h-\tau,\omega)-f_1(\cdot,t-\tau,\omega),\phi_n\rd\ \lambda_n^{1/2}\tau^{\alpha-1}
  E_{\alpha,\alpha}(-\lambda_n \tau^\alpha)~d\tau\Big]^2\\
   &\quad+C\sum_{n=1}^\infty \Big[\int_t^{t+h} \l f_1(\cdot,t+h-\tau,\omega),\phi_n\rd\ \lambda_n^{1/2}\tau^{\alpha-1}
   E_{\alpha,\alpha}(-\lambda_n \tau^\alpha)~d\tau\Big]^2\\
   &\quad+C\sum_{n=1}^\infty \Big[\int_0^t \l f_2(\cdot,t+h-\tau,\omega)-f_2(\cdot,t-\tau,\omega),\phi_n\rd 
   \ \lambda_n^{1/2}\tau^{\alpha-1}E_{\alpha,\alpha}(-\lambda_n \tau^\alpha)~dB(\tau)\Big]^2\\
   &\quad+C\sum_{n=1}^\infty \Big[\int_t^{t+h} \l f_2(\cdot,t+h-\tau,\omega),\phi_n\rd
   \ \lambda_n^{1/2}\tau^{\alpha-1}E_{\alpha,\alpha}(-\lambda_n \tau^\alpha)~dB(\tau)\Big]^2\\
   &=:C(S_1+S_2+S_3+S_4).
 \end{align*}

From Lemma \ref{mittag_bound} and the estimate \eqref{esti-useful1}, we see that 
\begin{equation}\label{inequality_2}
 |\lambda_n\tau^{2\alpha-2} E^2_{\alpha,\alpha}(-\lambda_n \tau^\alpha)|
 \le  \Big|\lambda_n\tau^{2\alpha-2} \frac{C}{(1+\lambda_n \tau^{\alpha})^2} \Big|
 \le C\tau^{\alpha-2},
\end{equation}
where the constant $C>0$ is independent of $n$. Further we can use the H\"older inequality and the Fubini lemma to deduce that 
\begin{equation*}
 \begin{aligned}
  \E[S_1]
  &\le \E\Big[\sum_{n=1}^\infty\big(\int_0^t \l f_1(\cdot,t+h-\tau,\omega)
  -f_1(\cdot,t-\tau,\omega),\phi_n\rd^2\ d\tau\big)
  \ \big(\int_0^t C\tau^{\alpha-2}\ d\tau \big)\Big]\\
  &\le Ct^{\alpha-1} \int_0^t \E\left[\| f_1(\cdot,t+h-\tau,\omega)-f_1(\cdot,t-\tau,\omega)\|^2_{L^2(D)} \right]\ d\tau\to 0, \quad \text{as}\ h\to 0\!+.
 \end{aligned}
\end{equation*}

For $S_2$, similarly we have  
\begin{equation*}
 \begin{aligned}
  \E[S_2]&\le\E\Big[\sum_{n=1}^\infty \big(\int_t^{t+h} \l f_1(\cdot,t+h-\tau,\omega),\phi_n\rd^2\ d\tau\big)
  \ \big(\int_t^{t+h}\lambda_n\tau^{2\alpha-2}E^2_{\alpha,\alpha}(-\lambda_n \tau^\alpha)~d\tau\big)\Big]\\
  &\le C\left[(t+h)^{\alpha-1}-t^{\alpha-1} \right]\int_t^{t+h} \| f_1(\cdot,t+h-\tau,\cdot)\|_{L^2(D\times\Omega)}^2\ d\tau\\
  &\le C\left[(t+h)^{\alpha-1}-t^{\alpha-1} \right]\ h \|f_1\|^2_{C([0,\infty);L^2(D\times\Omega))}\to 0, \quad \text{as}\ h\to 0\!+.
 \end{aligned}
\end{equation*}

For $S_3$, the Ito formula in Lemma \ref{Ito} yields  
\begin{equation*}
\begin{aligned}
 \E[S_3]
 =& \E\Big[\sum_{n=1}^\infty \int_0^t \l f_2(\cdot,t+h-\tau,\omega)-f_2(\cdot,t-\tau,\omega),\phi_n\rd^2
   \ \lambda_n\tau^{2\alpha-2}E^2_{\alpha,\alpha}(-\lambda_n \tau^\alpha)~d\tau \Big]\\
 \le& C \int_0^t \| f_2(\cdot,t+h-\tau,\cdot)-f_2(\cdot,t-\tau,\cdot)\|_{L^2(D\times\Omega)}^2   \ \tau^{\alpha-2}~d\tau.
\end{aligned}
\end{equation*}
From the facts $\alpha>1$ and $\|f_2\|_{C([0,\infty);L^2(D\times\Omega))}<\infty$, applying the dominated convergence theorem gives that 
\begin{equation*}
\begin{aligned}
 &\lim_{h\to0+}\int_0^t \| f_2(\cdot,t+h-\tau,\cdot)-f_2(\cdot,t-\tau,\cdot)\|_{L^2(D\times\Omega)}^2\ \tau^{\alpha-2}~d\tau\\
  &= \int_0^t \lim_{h\to0+}\| f_2(\cdot,t+h-\tau,\cdot)
  -f_2(\cdot,t-\tau,\cdot)\|_{L^2(D\times\Omega)}^2
   \ \tau^{\alpha-2}~d\tau=0,
\end{aligned}
\end{equation*}
which implies $\lim_{h\to0+} \E[S_3]=0$.

For $S_4$, using the similar argument for $S_3$, we get
\begin{equation*}
 \begin{aligned}
 \E[S_4] &=\sum_{n=1}^\infty \E\Big[\int_t^{t+h} \l f_2(\cdot,t+h-\tau,\omega),\phi_n\rd^2\ \lambda_n\tau^{2\alpha-2}E^2_{\alpha,\alpha}(-\lambda_n \tau^\alpha)\ d\tau \Big]\\
   &\le C \|f_2\|^2_{C([0,\infty);L^2(D\times\Omega))}\int_t^{t+h} 
   \tau^{\alpha-2}\ d\tau\to 0,\quad \text{as}\ h\to 0\!+.
 \end{aligned}
\end{equation*}

Collecting the estimates for $S_j,\ j=1,2,3,4$, it holds that    
\begin{equation*}
 \lim_{h\to0+} \| u(\cdot,t+h,\cdot)-u(\cdot,t,\cdot)\|^2_{L^2(\Omega;H^1(D))}=0, 
\end{equation*}
and the result for the case of $h\to 0-$ can be proved similarly. 
Hence, we have $u\in C([0,\infty);L^2(\Omega;H^1(D)))$.

Next, we deduce an estimate of $\|u\|_{C([0,T];L^2(\Omega;H^1(D)))}$.
 Fixing $t\in [0,T]\subset[0,\infty),$ the H\"older inequality, the Ito formula in Lemma \ref{Ito} and the previous arguments give that  
\begin{equation*}
 \begin{aligned}
 &\|u(\cdot,t,\cdot)\|^2_{L^2(\Omega;H^1(D))}\\
   &\le  C \sum_{n=1}^\infty  \Big[\int_0^t\lambda_n\tau^{2\alpha-2}E^2_{\alpha,\alpha}(-\lambda_n\tau^\alpha)~d\tau\Big]~
  \E\Big[\int_0^t \l f_1(\cdot,t-\tau,\omega),\phi_n\rd^2\ d\tau\Big]\\
  &\quad+C \sum_{n=1}^\infty \int_0^t\lambda_n\tau^{2\alpha-2} E^2_{\alpha,\alpha}(-\lambda_n\tau^\alpha) 
  \ \E\Big[\l f_2(\cdot,t-\tau,\omega),\phi_n\rd^2\Big]~d\tau\\
  &=:C(S_5+S_6).
  \end{aligned}  
 \end{equation*}
 With \eqref{inequality_2}, we see that 
 \begin{equation*}
 \begin{aligned}
  &S_5\le C t^{\alpha-1}\sum_{n=1}^\infty  \E\Big[\int_0^t \l f_1(\cdot,t-\tau,\omega),\phi_n\rd^2\ d\tau\Big]
  \le CT^{\alpha-1} \|f_1\|^2_{L^2(D\times(0,T)\times \Omega)},\\
  &S_6\le C \sum_{n=1}^\infty \int_0^t\tau^{\alpha-2} 
  \ \E\Big[\l f_2(\cdot,t-\tau,\omega),\phi_n\rd^2\Big]~d\tau
  \le C T^{\alpha-1}\|f_2\|^2_{C([0,\infty);L^2(D\times\Omega))}.
  \end{aligned}
 \end{equation*}
The estimates of $S_5$ and $S_6$ give the desired result and complete the proof. 
\end{proof}

\subsection{$L^2$-regularity.}

The $L^2$-regularity is discussed in this subsection. 
\begin{lemma}\label{lemma_regularity_l_source}
We suppose $f_1,f_2$ satisfy Assumption \ref{assumption}, $T\in (0,\infty)$ and $\alpha\in(0,1)\cup(1,2)$. Then we have the following $L^2$-regularity of $u$ in \eqref{equality_4},   
 \begin{equation*}
\|u\|_{L^2(D\times(0,T)\times\Omega)}
 \le C \big( T^{\alpha+\delta}\|f_1\|_{L^2(D\times(0,T)\times\Omega)}
 +T^{\alpha+\delta-1/2}\|f_2\|_{L^2(D\times(0,T)\times\Omega)}\big).
\end{equation*}
\end{lemma}
\begin{proof}
For each $\omega\in \Omega$, from the notation \eqref{equality_4}, a direct calculation yields
 \begin{equation*}
  \begin{aligned}
   \|u(\cdot,\cdot,\omega)\|^2_{L^2(D\times(0,T))}
   &=\int_0^T \|u(\cdot,t,\omega)\|^2_{L^2(D)} \ dt\\
   &=\int_0^T\sum_{n=1}^\infty [v_{n,1}(t,\omega)+v_{n,2}(t,\omega)]^2\ dt
   \\
   &\le C  \int_0^T\sum_{n=1}^\infty v_{n,1}^2(t,\omega)\ dt
   +C\int_0^T\sum_{n=1}^\infty v_{n,2}^2(t,\omega)\ dt\\
   &=C \sum_{n=1}^\infty\|v_{n,1}(\cdot,\omega)\|_{L^2(0,T)}^2
   +C\sum_{n=1}^\infty\|v_{n,2}(\cdot,\omega)\|_{L^2(0,T)}^2,
  \end{aligned}
 \end{equation*}
where the last equality is due to the monotone convergence theorem. The above estimate further implies that
\begin{equation*}
\begin{aligned}
 \|u\|^2_{L^2(D\times(0,T)\times\Omega)}
 &\le C \sum_{n=1}^\infty \|v_{n,1}\|_{L^2((0,T)\times \Omega)}^2 
 +C\sum_{n=1}^\infty\|v_{n,2}\|_{L^2((0,T)\times \Omega)}^2\\
 &=: C(S_1+S_2).
 \end{aligned}
\end{equation*}
Since $v_{n,1}(t,\o)=\left(t^{\alpha+\delta-1}E_{\alpha,\alpha+\delta}(-\lambda_n t^\alpha) \right)* \l f_1(\cdot,t,\omega),\phi_n\rd,$  
then by Young's convolution inequality, we see that
\begin{equation*}
 \|v_{n,1}(\cdot,\omega)\|_{L^2(0,T)}^2
 \le \|t^{\alpha+\delta-1}E_{\alpha,\alpha+\delta}(-\lambda_n t^\alpha)\|_{L^1(0,T)}^2
 \ \|\l f_1(\cdot,t,\omega),\phi_n\rd\|_{L^2(0,T)}^2. 
\end{equation*} 
Moreover, from Lemma \ref{mittag_bound}, it follows that   
\begin{equation*}
  \|t^{\alpha+\delta-1}E_{\alpha,\alpha+\delta}(-\lambda_n t^\alpha)\|_{L^1(0,T)} \le C\|t^{\alpha+\delta-1}\|_{L^1(0,T)}
  \le C T^{\alpha+\delta},
\end{equation*}
which combined with the monotone convergence theorem gives 
\begin{equation*}
\begin{aligned}
 S_1&\le CT^{2\alpha+2\delta}\int_\Omega \sum_{n=1}^\infty  
 \|\l f_1(\cdot,t,\omega),\phi_n\rd\|_{L^2(0,T)}^2 \ dP(\o)\\
 &=CT^{2\alpha+2\delta}\int_\Omega\int_0^T \sum_{n=1}^\infty  
 \l f_1(\cdot,t,\omega),\phi_n\rd^2 \ dt\ dP(\o)\\
 &=CT^{2\alpha+2\delta} \|f_1\|_{L^2(D\times(0,T)\times\Omega)}^2.
 \end{aligned}
\end{equation*}

For $S_2,$ we have  
\begin{equation*}
 \begin{aligned}
  S_2&=\sum_{n=1}^\infty\|v_{n,2}\|_{L^2((0,T)\times \Omega)}^2
  =\sum_{n=1}^\infty \int_0^T \int_\Omega v_{n,2}^2(t,\omega)
  \ dP(\o) \ dt.
 \end{aligned}
\end{equation*} 
Then Lemma \ref{mittag_bound} and the Ito formula in Lemma \ref{Ito}  give that
\begin{equation*}
\begin{aligned}
 \int_\Omega v_{n,2}^2(t,\omega)\ dP(\o)
 &=\E\Big[ \int_0^t \tau^{2\alpha+2\delta-2}E_{\alpha,\alpha+\delta}^2(-\lambda_n\tau^\alpha) 
 \ \l f_2(\cdot,t-\tau,\omega),\phi_n\rd^2\ d\tau \Big]\\
 &\le C \E\Big[ \int_0^t \tau^{2\alpha+2\delta-2}
 \ \l f_2(\cdot,t-\tau,\omega),\phi_n\rd^2\ d\tau \Big];
 \end{aligned}
\end{equation*}
meanwhile using Young's convolution inequality again we have 
\begin{equation*}
\begin{aligned}
 S_2&\le C\int_0^T \int_0^t \tau^{2\alpha+2\delta-2}\ \E\Big[\sum_{n=1}^\infty\l f_2(\cdot,t-\tau,\omega),\phi_n\rd^2 \Big]\ d\tau\ dt  \\
  &\le C\|t^{2\alpha+2\delta-2} \|_{L^1(0,T)}\ \int_0^T \E\Big[\sum_{n=1}^\infty\l f_2(\cdot,t,\omega),\phi_n\rd^2 \Big]\ dt\\
  &\le C T^{2\alpha+2\delta-1}\|f_2\|^2_{L^2(D\times(0,T)\times\Omega)}.
  \end{aligned}
\end{equation*}

By the estimates for $S_1$ and $S_2$, we finally obtain
\begin{equation*}
 \|u\|_{L^2(D\times(0,T)\times\Omega)}
 \le C \big( T^{\alpha+\delta}\|f_1\|_{L^2(D\times(0,T)\times\Omega)}
 +T^{\alpha+\delta-1/2}\|f_2\|_{L^2(D\times(0,T)\times\Omega)}\big),
\end{equation*}
and complete the proof.
\end{proof}

When $\alpha\in(1/2,1)$, from condition \eqref{condition_delta}, we can set $\delta=0$. In this case, we can extend the $L^2$-regularity in Lemma \ref{lemma_regularity_l_source} to $D\times(0,\infty)\times \Omega$. 
\begin{corollary}\label{infty}
For the case of $\alpha\in(1/2,1)$, if stronger conditions $f_j\in L^2(D\times(0,\infty)\times\Omega)$ for $j=1,2$ are provided, then $u$ in \eqref{equality_4} satisfies the following regularity, 
 \begin{equation*}
 \|u\|_{L^2(D\times(0,\infty)\times\Omega)}\le C \big( \|f_1\|_{L^2(D\times(0,\infty)\times\Omega)}
 +\|f_2\|_{L^2(D\times(0,\infty)\times\Omega)}\big).
\end{equation*}
\end{corollary}
\begin{proof}
 The proof of Lemma \ref{lemma_regularity_l_source} gives that 
 \begin{equation*}
\begin{aligned}
 \|u\|^2_{L^2(D\times(0,\infty)\times\Omega)}
 &\le C \sum_{n=1}^\infty \|v_{n,1}\|_{L^2((0,\infty)\times \Omega)}^2 
 +C\sum_{n=1}^\infty\|v_{n,2}\|_{L^2((0,\infty)\times \Omega)}^2\\
 &=: C(S_1+S_2),
 \end{aligned}
\end{equation*}
and 
\begin{equation*}
\begin{aligned}
 S_1&\le  \sum_{n=1}^\infty \|t^{\alpha-1}E_{\alpha,\alpha}(-\lambda_n t^\alpha)\|_{L^1(0,\infty)}^2
 \ \|\l f_1(\cdot,t,\omega),\phi_n\rd\|_{L^2((0,\infty)\times\Omega)}^2,\\
 S_2&\le \sum_{n=1}^\infty \|t^{2\alpha-2}E_{\alpha,\alpha}^2(-\lambda_nt^\alpha)\|_{L^1(0,\infty)}\ \|\l f_2(\cdot,t,\omega),\phi_n\rd\|_{L^2((0,\infty)\times\Omega)}^2.
\end{aligned}
\end{equation*}
Since $\alpha\in (1/2,1)$, Lemma \ref{mittag_monotone} can be applied and it yields that $E_{\alpha,\alpha}(-\lambda_n t^\alpha)\ge 0$. This with Lemma \ref{mittag_derivative} leads to  
\begin{equation*}
\begin{aligned}
  \|t^{\alpha-1}E_{\alpha,\alpha}(-\lambda_n t^\alpha)\|_{L^1(0,\infty)}
 &=\int_0^\infty t^{\alpha-1}E_{\alpha,\alpha}(-\lambda_n t^\alpha)\ dt\\
 &=[-\lambda_n^{-1}E_{\alpha,1}(-\lambda_n t^\alpha)]\big|_0^\infty
 \le \lambda_1^{-1}<\infty.
 \end{aligned}
\end{equation*}
Also, \eqref{inequality_1} with $\delta=0$ gives that $$\|t^{2\alpha-2}E_{\alpha,\alpha}^2(-\lambda_nt^\alpha)\|_{L^1(0,\infty)}\le C<\infty,$$ and the constant $C>0$ does not depend on $n$. Hence, we have  
\begin{equation*}
S_1\le C\|f_1\|_{L^2(D\times(0,\infty)\times\Omega)}^2,
\ \ S_2\le C \|f_2\|_{L^2(D\times(0,\infty)\times\Omega)}^2,
\end{equation*}
and complete the proof.
\end{proof}

If $f_2$ is imposed a stronger assumption $f_2\in L^2((0,T)\times \Omega;H^1(D))$, the smoothness of $u$ will be promoted. 
\begin{lemma}\label{lemma_regularity_h_source}
Given the same assumptions in Lemma \ref{lemma_regularity_l_source} and $f_2\in L^2((0,T)\times\Omega;H^1(D))$, the function $u$ in \eqref{equality_4} satisfies 
 \begin{equation*}
  \|u\|_{L^2((0,T)\times \Omega;H^1(D))}\le C\big(T^{\alpha/2+\delta}\|f_1\|_{L^2(D\times(0,T)\times\Omega)}+\|f_2\|_{L^2((0,T)\times \Omega;H^1(D))}\big).
 \end{equation*}
\end{lemma}
\begin{proof}
By virtue of the relation $D(\mathcal A^{1/2}) = H_0^1(D)$, a direct calculation gives that  
 \begin{equation*}
  \begin{aligned}
   \|u(\cdot,\cdot,\omega)\|^2_{L^2(0,T;H^1(D))}
     &\le C\int_0^T\|\A^{1/2} u(\cdot,t,\omega)\|^2_{L^2(D)}\ dt\\
   &\le C  \int_0^T\sum_{n=1}^\infty \lambda_nv_{n,1}^2(t,\o)\ dt
   +C\int_0^T\sum_{n=1}^\infty \lambda_nv_{n,2}^2(t,\omega)\ dt\\
   &=C \sum_{n=1}^\infty\|\lambda_n^{1/2}v_{n,1}(\cdot,\o)\|_{L^2(0,T)}^2
   +C\sum_{n=1}^\infty\|\lambda_n^{1/2}v_{n,2}(\cdot,\omega)\|_{L^2(0,T)}^2,
  \end{aligned}
 \end{equation*}
 which leads to  
\begin{equation*}
\begin{aligned}
\|u\|^2_{L^2((0,T)\times \Omega;H^1(D))}
 &\le C \sum_{n=1}^\infty\|\lambda_n^{1/2}v_{n,1}\|_{L^2((0,T)\times \Omega)}^2
 +C\sum_{n=1}^\infty\|\lambda_n^{1/2}v_{n,2}\|_{L^2((0,T)\times \Omega)}^2\\
 &=: C(S_1+S_2).
 \end{aligned}
\end{equation*}

We next estimate the terms $S_1$ and $S_2$ separately. Firstly we conclude from Lemma \ref{lem-ml-int} and the inequality \eqref{esti-useful1} that  
\begin{equation*}
\begin{aligned}
  \|\lambda_n^{1/2}t^{\alpha+\delta-1}E_{\alpha,\alpha+\delta}(-\lambda_n t^\alpha)\|_{L^1(0,T)}
 \le C\int_0^T t^{\alpha/2+\delta-1}\frac{\lambda_n^{1/2} t^{\alpha/2}}{1+\lambda_n t^\alpha}\ dt
 \le CT^{\alpha/2+\delta}.
 \end{aligned}
\end{equation*}
This with Young's convolution inequality yields
\begin{equation*}
\begin{aligned}
 S_1&\le \int_\Omega \sum_{n=1}^\infty \|\lambda_n^{1/2}t^{\alpha+\delta-1}
 E_{\alpha,\alpha+\delta} (-\lambda_n t^\alpha)\|_{L^1(0,T)}^2 \ 
 \|\l f_1(\cdot,t,\omega),\phi_n\rd\|_{L^2(0,T)}^2\ dP(\o)\\
 &\le CT^{\alpha+2\delta}\int_\Omega\sum_{n=1}^\infty \|\l f_1(\cdot,t,\omega),\phi_n\rd\|_{L^2(0,T)}^2 \ dP(\o)
 =CT^{\alpha+2\delta}\|f_1\|^2_{L^2(D\times(0,T)\times\Omega)}.
\end{aligned}
\end{equation*}
For the bound of $S_2,$ Young's convolution inequality, the Ito isometry formula in Lemma \ref{Ito} and \eqref{inequality_1} give that 
 \begin{align*}
  S_2&=\sum_{n=1}^\infty \int_0^T \E[\lambda_nv_{n,2}^2(\cdot,\omega)] \ dt\\
  &\le \sum_{n=1}^\infty \Big[\int_0^T t^{2\alpha+2\delta-2}E_{\alpha,\alpha+\delta}^2  (-\lambda_nt^\alpha)\ dt\Big]\ \E\Big[\int_0^T 
  \lambda_n\l f_2(\cdot,t,\omega),\phi_n\rd^2\ dt\Big]\\
   &\le C \E\Big[\int_0^T\| f_2(\cdot,t,\omega)\|_{H^1(D)}^2\ dt\Big]
   = C \|f_2\|^2_{L^2((0,T)\times \Omega;H^1(D))}.
 \end{align*}
The estimates of $S_1$ and $S_2$ give the desired result and complete the proof.
\end{proof}

Thanks to the properties of the Mittag-Leffler function in Lemmas \ref{mittag_monotone} and \ref{mittag_derivative}, a sharp spatial regularity of the function $u$ in \eqref{equality_4} can be achieved, with the conditions 
$\alpha\in(1/2,1)$, $f_1\in L^2(D\times(0,\infty)\times\Omega)$ 
and $f_2\in L^2((0,\infty)\times \Omega;H^2(D))$.
\begin{corollary}\label{infty_h} 
Provided 
$$\alpha\in(1/2,1),\ f_1\in L^2(D\times(0,\infty)\times\Omega),\ f_2\in L^2((0,\infty)\times \Omega;H^2(D)),$$ 
the following regularity estimate of $u$ in \eqref{equality_4} holds,  
 \begin{equation*}
  \|u\|_{L^2((0,\infty)\times \Omega;H^2(D))}\le C\big(\|f_1\|_{L^2(D\times(0,\infty)\times\Omega)}
  + \|f_2\|_{L^2((0,\infty)\times \Omega;H^2(D))}\big).
 \end{equation*}
\end{corollary}
\begin{proof}
Since $\alpha\in(1/2,1)$, we can set $\delta=0$ in \eqref{equality_4}. Similar to the proof of Lemma \ref{lemma_regularity_h_source}, we have
 \begin{equation*}
\begin{aligned}
 \|u\|^2_{L^2((0,\infty)\times \Omega;H^2(D))}
 &\le C \sum_{n=1}^\infty\|\lambda_nv_{n,1}\|_{L^2((0,\infty)\times \Omega)}^2
 +C\sum_{n=1}^\infty\|\lambda_nv_{n,2}\|_{L^2((0,\infty)\times \Omega)}^2
 \\&=: C(S_1+S_2).
 \end{aligned}
\end{equation*}
With Lemma \ref{mittag_monotone}, we have $E_{\alpha,\alpha} (-\lambda_n t^\alpha)\ge 0$, which combined with Lemma \ref{mittag_derivative} yields
\begin{equation*}
\begin{aligned}
 \|\lambda_nt^{\alpha-1}E_{\alpha,\alpha} (-\lambda_n t^\alpha)\|_{L^1(0,\infty)}
 &=\int_0^\infty \lambda_nt^{\alpha-1}E_{\alpha,\alpha} (-\lambda_n t^\alpha)\ dt\\
 &=[-E_{\alpha,1}(-\lambda_n t^\alpha)]\big|_0^\infty=1.
 \end{aligned}
\end{equation*}
Then applying Young's convolution inequality gives that  
\begin{equation*}
\begin{aligned}
 S_1&\le \int_\Omega \sum_{n=1}^\infty \|\lambda_nt^{\alpha-1}
 E_{\alpha,\alpha} (-\lambda_n t^\alpha)\|_{L^1(0,\infty)}^2 \ 
 \|\l f_1(\cdot,t,\omega),\phi_n\rd\|_{L^2(0,\infty)}^2\ dP(\o)\\
 &=\int_\Omega\sum_{n=1}^\infty \|\l f_1(\cdot,t,\omega),\phi_n\rd\|
 _{L^2(0,\infty)}^2 \ dP(\o)
 =\|f_1\|^2_{L^2(D\times(0,\infty)\times\Omega)}.
\end{aligned}
\end{equation*}

For the bound of $S_2$, we conclude from the Young's convolution inequality and \eqref{inequality_1} that 
\begin{equation*}
 \begin{aligned}
  S_2&=\sum_{n=1}^\infty \int_0^\infty \E[\lambda_n^2 v_{n,2}^2(t,\o)]  \ dt\\
  &\le \sum_{n=1}^\infty \Big[\int_0^\infty t^{2\alpha-2}E_{\alpha,\alpha}^2  (-\lambda_nt^\alpha)\ dt\Big]\ \E\Big[\int_0^\infty
  \lambda_n^2\l f_2(\cdot,t,\omega),\phi_n\rd^2\ dt\Big]\\
   &\le C \|f_2\|^2_{L^2((0,\infty)\times \Omega;H^2(D))}.
 \end{aligned}
\end{equation*}
The estimates of $S_1$ and $S_2$ give the desired result and complete the proof.
\end{proof}

\begin{remark}
Here we should point out that the lack of monotonic property of the Mittag-Leffler function $E_{\alpha,\beta}(-t)$ with $\alpha\in(1,2)$ causes several inconveniences in analyzing the regularity of the solution. For instance, for the $L^2$-regularity of the solution when $\alpha\in(1,2)$, $T$ cannot be infinity even if $f_1$ and $f_2$ have much stronger spatial  smoothness. 
\end{remark}

\subsection{Existence and uniqueness.}
Finally we consider the existence and uniqueness of solution \eqref{equality_4}. 
\begin{lemma}\label{lemma_existence_source}
 Under Definition \ref{def-solution}, the problem \eqref{equ-gov} admits
 a unique stochastic weak solution with the representation \eqref{equality_4}.
\end{lemma}
\begin{proof}
 Following the proof of Lemma \ref{lemma_existence_initial}, and setting $v_n(t,\o)=\l u(\cdot,t,\o),\phi_n\rd$ for $n\in\mathbb N^+$ and a.e. $\o\in\Omega$, the corresponding equation is deduced as 
\begin{equation}\label{ODE_2}
 v_n(t,\o)+ \lambda_n I_t^\alpha v_n(t,\o)=
 \langle I_t^\alpha F(\cdot,t,\o),\phi_n \rangle_{L^2(D)}.
\end{equation}
 We denote $\tilde v_n(t,\o)=v_{n,1}(t,\o)+v_{n,2}(t,\o)$. From the definition of the Riemann-Liouville fractional derivative and the Fubini lemma, it holds that 
 \begin{align*}
  &\lambda_n\I v_{n,1}(t,\o)\\
  &=\frac{\lambda_n}{\Gamma(\alpha)}\int_0^t (t-\tau)^{\alpha-1}  
   \Big[ \int_0^\tau \l f_1(\cdot,\tau-s,\omega),\phi_n\rd\ s^{\alpha+\delta-1}
   E_{\alpha,\alpha+\delta} (-\lambda_n s^\alpha)\ ds\Big]\ d\tau\\
   &=\frac{\lambda_n}{\Gamma(\alpha)}\int_0^t\l f_1(\cdot,s,\omega),\phi_n\rd 
    \Big[ \int_s^t (t-\tau)^{\alpha-1} (\tau-s)^{\alpha+\delta-1}
   E_{\alpha,\alpha+\delta} (-\lambda_n (\tau-s)^\alpha)\ d\tau\Big]\ ds.
 \end{align*}
 Moreover, by noting the equality $I_t^\alpha t^\gamma =\Gamma(\gamma+1)t^{\alpha+\gamma}/\Gamma(\alpha+\gamma+1)$ with $\alpha>0$ and $\gamma>-1$, from the series definition of the Mittag-Leffler function we see that
 \begin{equation*}
  \begin{aligned}
 \lambda_nI_t^\alpha \big( t^{\alpha+\delta-1} E_{\alpha,\alpha+\delta}(-\lambda_n t^\alpha)\big) 
 =& -t^{\alpha + \delta -1} \sum_{k=0}^\infty \frac{(-\lambda_n t^\alpha)^{k+1}}{\Gamma(\alpha k + 2\alpha +\delta)}\\
 =& t^{\alpha+\delta-1}
   \left[\Gamma(\alpha+\delta)^{-1}-E_{\alpha,\alpha+\delta}(-\lambda_n t^\alpha)\right],
 \end{aligned}
 \end{equation*}
 which gives 
\begin{equation}\label{e_3}
 \begin{aligned}
  &\lambda_n\I v_{n,1}(t,\o)\\
   &=\int_0^t\l f_1(\cdot,s,\omega),\phi_n\rd\ (t-s)^{\alpha+\delta-1}
   \left[\Gamma(\alpha+\delta)^{-1}-E_{\alpha,\alpha+\delta}(-\lambda_n(t-s)^\alpha)\right]\ ds\\
   &= -v_{n,1}(t,\omega)+\Gamma(\alpha+\delta)^{-1}\int_0^t(t-s)^{\alpha+\delta-1}\l f_1(\cdot,s,\omega),\phi_n\rd\ ds.
 \end{aligned}
\end{equation}
Analogously, we can see that
\begin{equation}\label{e_4}
   \begin{aligned}
   &\lambda_n\I v_{n,2}(t,\o)\\
   &=\frac{1}{\Gamma(\alpha)}\int_0^t (t-\tau)^{\alpha-1}  
   \int_0^\tau \l f_2(\cdot,\tau-s,\omega),\phi_n\rd\ s^{\alpha+\delta-1}
   E_{\alpha,\alpha+\delta} (-\lambda_n s^\alpha)~dB(s)\ d\tau\\
   &=\frac{1}{\Gamma(\alpha)}\int_0^t \l f_2(\cdot,s,\omega),\phi_n\rd  
   \int_s^t (t-\tau)^{\alpha-1} (\tau-s)^{\alpha+\delta-1}
   E_{\alpha,\alpha+\delta} (-\lambda_n (\tau-s)^\alpha)~d\tau\ dB(s)\\
   &=\int_0^t \l f_2(\cdot,s,\omega),\phi_n\rd\ (t-s)^{\alpha+\delta-1}
   [\Gamma(\alpha+\delta)^{-1}-E_{\alpha,\alpha+\delta}(-\lambda_n(t-s)^\alpha)]\ dB(s)\\
   &=-v_{n,2}(t,\omega) + \Gamma(\alpha+\delta)^{-1}\int_0^t(t-s)^{\alpha+\delta-1}\l f_2(\cdot,s,\omega),\phi_n\rd\ dB(s).
  \end{aligned}
  \end{equation}
These imply that $\tilde v_n$ satisfies equation \eqref{ODE_2}, which together with the uniqueness in \cite[Theorem 3.25]{KilbasSrivastavaTrujillo:2006} yields $v_n=\tilde v_n$. With the completeness of $\{\phi_n\}_{n=1}^\infty$ in $H^2(D)\cap H_0^1(D)$, 
the uniqueness is obtained.

Again recalling $\mathcal A\phi_n=\lambda_n\phi_n$ and $\{\phi_n\}_{n=1}^\infty$ is an orthonormal basis of $L^2(D)$, we sum up equations \eqref{e_3} and \eqref{e_4} for each $n\in\mathbb N^+$ with multiplying $\l\varphi,\phi_n \rd$. Then we have that the series  $u=\sum_{n=1}^\infty (v_{n,1}+v_{n,2}) \phi_n$ satisfies equation \eqref{equ-var}. With the continuity ensured by Lemma \ref{lemma_continuity}, it holds that $u$ given by \eqref{equality_4} is the stochastic weak solution of the problem \eqref{equ-gov} under Definition \ref{def-solution}. The proof is complete. 
\end{proof}

Now we can complete the proof of Theorem \ref{thm-source}. 
\begin{proof}[Proof of Theorem \ref{thm-source}]
 The existence and uniqueness of the stochastic weak solution follow from Lemma \ref{lemma_existence_source}, and the regularities are given by Lemmas \ref{lemma_continuity}, 
 \ref{lemma_regularity_c_source}, \ref{lemma_regularity_l_source}, \ref{lemma_regularity_h_source} and Corollaries \ref{cor-wave-source-conti}, \ref{infty}, \ref{infty_h}.  
\end{proof}

\section{Application in inverse problems.}\label{sec-isp}

In this section, $\Gamma$ is set to be an open subset of the boundary $\partial D$. Also we assume the source term $F$ in the problem \eqref{equ-gov} is of the form 
\begin{equation}\label{eq-source}
F(x,t)= I_t^\delta \Big(f_1(x) g_1(t) + f_2(x) g_2(t) \frac{dB(t)}{dt} \Big).
\end{equation}
Here we recall that $\delta>0$ satisfies \eqref{condition_delta} and assume that $g_i\not\equiv0,\ i=1,2$ are known $C^1$-functions. The interested inverse problem is to determine $f_1(x)$ and $f_2(x)$ in \eqref{eq-source} from the moments of the partial boundary measurements  
\begin{equation}
\label{ob-u}
\partial_{\nu_{\mathcal A}}u(x,t,\omega):=\sum_{i,j=1}^d a_{ij}(x)\nu_j\frac{\partial u}{\partial x_i}(x,t,\omega) ,\quad (x,t,\omega) \in \Gamma\times (0,T)\times \Omega,
\end{equation}
where $(\nu_1, \nu_2, \cdots, \nu_d)$ is the unit outward normal vector at the boundary $\partial D$. In this section we will establish the uniqueness of this inverse source problem.

\subsection{Literature.}
In the absence of random source terms, we can find some positive answers for the uniqueness topic. For example, Zhang and Xu \cite{Ying2011Inverse} and Li, Liu and Yamamoto \cite{LiLiuYa2022} discussed the uniqueness of the inverse source problem for the $1$D time-fractional diffusion equation; Jiang, Li, Liu and Yamamoto \cite{Jiang2017Weak} established the results parallel to \cite{Ying2011Inverse} for the multi-dimensional case by constructing a weak type unique continuation, and proposed the iteration algorithm for the numerical treatment. For some other works about the fractional inverse problems, we refer to  \cite{Cheng2009Uniqueness,Kirane2011Determination,Li2015Uniqueness,Li2020Unique,Liu2016Strong,Sakamoto2013Inverse,SunLiu2020,WSL16,Zhidong2016An} and the references therein. 

For the case of recovering the random source, we collect several representative works. 
In \cite{feng2020inverse} and \cite{nie2021inverse}, the recovery of spatially varying source terms from final overdetermination data was investigated. By the properties and the Ito isometry of Brownian motions, the authors showed the uniqueness and stability of the inverse random source problem. The articles  \cite{fu2021application,gong2021numerical,liu2019reconstruction,niu2020inverse} provided the well-posedness of the direct problem and established the stability of the inverse problem for  determining the time-dependent sources. The main idea is to use the Duhamel principle and strong maximum principle. Moreover, \cite{gong2021numerical,liu2019reconstruction,niu2020inverse}  developed several effective reconstruction algorithms.

Unlike most of the existing works, which solve the stochastic inverse problems by the final overdetermination data \cite{feng2020inverse, nie2021inverse} or the interior point data \cite{fu2021application, gong2021numerical, liu2019reconstruction, niu2020inverse}, this work 
uses the moments of the partial boundary measurements. As a result, we need to overcome some novel difficulties in the proofs. Firstly, the random terms makes our discussion more difficult. Indeed, the solution of the equation can only achieve weaker temporal regularity estimates than the deterministic case. Then it is challenging to deduce an integral identity relation between  the changes of data and unknowns (see Lemma \ref{lem-duh}). On the other hand, in view of the setting of the source \eqref{eq-source}, there are two unknowns $f_i$, $i=1,2$ that need to be recovered, which makes the usual analytical techniques in, e.g., Jiang et al. \cite{Jiang2017Weak} do not work directly. Specifically, it is difficult to establish the Duhamel principle analogue to that in \cite{Jiang2017Weak}. Using the expectation and variance of the observed data, we decompose the original inverse source problem into the unique continuation problem of two corresponding homogeneous equations and then prove the uniqueness of the inverse source problem. 

\subsection{Uniqueness theorem.}
The following is the uniqueness theorem of this inverse random source problem. 
\begin{theorem}\label{thm-isp}
We let $f_1\in H_0^1(D)$, $f_2\in L^2(D)$ and $g_i\in C^1[0, T]$ satisfy $g_i(0) \ne 0$ for $i = 1,2$. Let $u\in L^2((0,T)\times\Omega; H^2(D)\cap H_0^1(D)) \cap C([0,T];L^2(D\times\Omega))$  be the stochastic weak solution to the problem \eqref{equ-gov} with $u_0=0$, $u_1=0$ if $\alpha\in(1,2)$. Then 
$$
\mathbb E[\partial_{\nu_{\mathcal A}} u(x,t,\o)] = \mathrm{Var} [\partial_{\nu_{\mathcal A}} u(x,t,\o) ]=0,\ (x,t)\in \Gamma \times (0, T), 
$$
implies $f_1 = f_2 = 0$ in $D$. 
\end{theorem}

Before giving the proof of the above theorem, we first show several useful lemmas which are related to Riemann-Liouville integral and the representation formula of the solution $u$. The first one is about the convolution formula for the Riemann-Liouville fractional integral.
\begin{lemma}
\label{lem-convo-RL}
Let $\alpha>0$ and $g,h\in L^2(0,T)$, then 
\begin{align*}
 I_t^\alpha\Big(\int_0^t g(\tau) h(t-\tau)\ d\tau \Big)
&=\int_0^t  g(\tau) (I_t^\alpha h)(t-\tau)\ d\tau,\\
I_t^\alpha\Big(\int_0^t g(\tau) h(t-\tau)\ dB(\tau) \Big)
&=\int_0^t  g(\tau) (I_t^\alpha h)(t-\tau)\ dB(\tau),\quad 0<t<T.
\end{align*}
\end{lemma}
\begin{proof}
We only give the proof of the second equality since the first one can be derived analogously.  From the definition of the Riemann-Liouville fractional integral $I_t^\alpha$, it follows that
\begin{align*}
I_t^\alpha \Big(\int_0^t g(\tau) h(t-\tau)\ dB(\tau) \Big)
=&\frac1{\Gamma(\alpha)} \int_0^t (t-\eta)^{\alpha-1} \Big[ \int_0^\eta g(\tau) h(\eta-\tau) \ dB(\tau) \Big]\ d\eta\\
=&\frac1{\Gamma(\alpha)} \int_0^t g(\tau) \Big[ \int_\tau^t (t-\eta)^{\alpha - 1} h(\eta-\tau) \ d\eta  \Big]\ dB(\tau).
\end{align*}
Here in the last equality the Fubini lemma is used. After the change of the variable, we find
\begin{align*}
I_t^\alpha \Big(\int_0^t g(\tau) h(t-\tau)\ dB(\tau) \Big)
=& \frac1{\Gamma(\alpha)} \int_0^t g(t-\tau)  \Big[ \int_0^\tau (\tau-\eta)^{\alpha -1} h(\eta) \ d\eta \Big]\ dB(\tau)\\
=& \int_0^t g(t-\tau) I_t^\alpha h(\tau)\ dB(\tau).
\end{align*}
We complete the proof of the lemma.
\end{proof}
Recalling the representation formula \eqref{equality_4} of the solution, we can derive the following Duhamel principle, which allows us to represent the solution of the nonhomogeneous equation via the corresponding homogeneous one.
\begin{lemma}\label{lem-duh}
Under the same assumptions in Theorem \ref{thm-isp}, we set $\delta>0$ being a constant such that \eqref{condition_delta} and let $u$ be the solution of model \eqref{equ-gov} with $u_0=0$, $u_1=0$ if $\alpha\in(1,2)$. Then $u(x,t,\o)$ can be represented as 
\begin{equation}
\label{eq-duh}
\begin{aligned}
u(x,t,\o)
= \int_0^t g_1(t-\tau) I_t^{\alpha+\delta-1}v_1(x,\tau)\ d\tau + \int_0^t g_2(t-\tau) I_t^{ \alpha+\delta-1}v_2(x,\tau)\ dB(\tau)
  \end{aligned}
 \end{equation}
 in the case of $\alpha+\delta - 1\ge 0$. If $\alpha+\delta-1<0$, we have 
 \begin{equation}
\label{eq-duh2}
\begin{aligned}
I_t^{1-\alpha-\delta} u(x,t,\o)
= \int_0^t g_1(t-\tau) v_1(x,\tau)\ d\tau + \int_0^t g_2(t-\tau) v_2(x,\tau)\ dB(\tau),
  \end{aligned}
 \end{equation}
 where $v_i$, $i=1,2$ are the solutions of the following homogeneous initial-boundary value problems
\begin{equation}\label{eq-homo}
\begin{cases}
\begin{aligned}
\partial_t^\alpha v_i(x,t) - \Delta v_i(x,t) &= 0 &&\text{ in } D\times(0,T),\\
v_i(x,t) &= f_i(x) &&\text{ in } D\times\{0\},\\
\partial_t v_i(x,t)&=0 && \text{ in } D\times\{0\},\mbox{ if } \alpha\in(1,2),\\
v_i(x,t) &= 0 &&\text{ on } \partial D\times (0, T).
\end{aligned}
\end{cases}
\end{equation}
\end{lemma}
\begin{proof}
Firstly, according to Theorem \ref{thm-wave-initial}, it is not difficult to check that the solutions $v_i$, $i=1,2$ to the homogeneous initial-boundary value problem \eqref{eq-homo} admit the following Fourier expansion
\begin{equation}\label{eq-v-f}
\begin{aligned}
v_i(x,t)
=&\sum_{n=1}^\infty \l f_i,\phi_n \rd E_{\alpha,1} (-\lambda_n t^\alpha) \phi_n(x),\quad i=1,2.
  \end{aligned}
 \end{equation}
On the other hand, in view of Theorem \ref{thm-source} with $F$ being of the form \eqref{eq-source}, we see that $u(x,t,\o)$ can be further rephrased  as follows
 \begin{equation}\label{eq-u-f}
 \begin{aligned}
  u(x,t,\o)=&\sum_{n=1}^\infty  \Big[\l f_1,\phi_n \rd \int_0^t  g_1(t-\tau) \tau^{\alpha+\delta-1}E_{\alpha,\alpha+\delta} (-\lambda_n \tau^\alpha)\ d\tau \\
  &\qquad+\l f_2,\phi_n\rd \int_0^t g_2(t-\tau) \tau^{\alpha+\delta-1}E_{\alpha,\alpha+\delta}
  (-\lambda_n \tau^\alpha)\ dB(\tau)\Big]\ \phi_n(x).
  \end{aligned}
 \end{equation}
Now if $\alpha+\delta-1\ge0$, in view of \cite[$(1.82)$]{Podlubny1999fractional}, we see that
  $$
 I_t^{\alpha+\delta-1} E_{\alpha,1} (-\lambda_n t^\alpha) = t^{\alpha+\delta-1}E_{\alpha,\alpha+\delta} (-\lambda_n t^\alpha).
 $$
Substituting the above equality into the representation formula \eqref{eq-u-f}, we arrive at
 \begin{align*}
  u(x,t,\o)=&\sum_{n=1}^\infty \Big[\l f_1,\phi_n \rd  \int_0^t  g_1(t-\tau)  I_\tau^{\alpha+\delta-1} E_{\alpha,1} (-\lambda_n \tau^\alpha) \ d\tau \\
  &\qquad+\l f_2,\phi_n\rd \int_0^t g_2(t-\tau)   I_\tau^{\alpha+\delta-1} E_{\alpha,1} (-\lambda_n \tau^\alpha)  \ dB(\tau)\Big]\ \phi_n(x).
  \end{align*}
By the dominated convergence theorem, we can exchange the infinite summation and the integral in the right hand side of the above equality to derive that
 \begin{equation*}
 \begin{aligned}
  u(x,t,\o)=& \int_0^t g_1(t-\tau) I_\tau^{\alpha+\delta-1} \Big(\sum_{n=1}^\infty \l f_1,\phi_n \rd   E_{\alpha,1} (-\lambda_n \tau^\alpha) \phi_n(x) \Big)\ d\tau \\
  &+ \int_0^t g_2(t-\tau) I_\tau^{\alpha+\delta-1} \Big(\sum_{n=1}^\infty \l f_2,\phi_n \rd   E_{\alpha,1} (-\lambda_n \tau^\alpha) \phi_n(x) \Big)\ dB(\tau).
  \end{aligned}
 \end{equation*}
Finally, we can see that the equality \eqref{eq-duh} is valid by noting the series representation formula \eqref{eq-v-f}. It remains to show \eqref{eq-duh2} in the case of $\alpha+\delta-1<0$. For this, we multiply the Riemann-Liouville fractional integral $I_t^{1-\alpha-\delta}$ on both sides of the representation formula \eqref{eq-u-f}. Then by Lemma \ref{lem-convo-RL}, we find
\begin{align*}
 & I_t^{1-\alpha-\delta} u(x,t,\o) \\
 &=\sum_{n=1}^\infty \Big[\l f_1,\phi_n \rd  \int_0^t  g_1(t-\tau)  I_\tau^{1-\alpha-\delta} \big( \tau^{\alpha+\delta-1}E_{\alpha,\alpha+\delta} (-\lambda_n \tau^\alpha)\big)\ d\tau \\
  &\quad\qquad\ +\l f_2,\phi_n\rd \int_0^t g_2(t-\tau)  I_\tau^{1-\alpha-\delta} \big( \tau^{\alpha+\delta-1}E_{\alpha,\alpha+\delta} (-\lambda_n \tau^\alpha)\big)\  dB(\tau)\Big]\ \phi_n(x).
  \end{align*}
Moreover, again by virtue of \cite[$(1.82)$]{Podlubny1999fractional}, we see that 
  $$
 I_t^{1-\alpha-\delta} \Big(t^{\alpha+\delta-1}E_{\alpha,\alpha+\delta} (-\lambda_n t^\alpha) \Big)= E_{\alpha,1} (-\lambda_n t^\alpha), 
 $$
 from which $I_t^{1-\alpha-\delta}u(x,t,\o)$ can be further represented by
\begin{align*}
 I_t^{1-\alpha-\delta} u(x,t,\o)=&\sum_{n=1}^\infty\Big[ \l f_1,\phi_n \rd  \int_0^t  g_1(t-\tau)  E_{\alpha,1} (-\lambda_n \tau^\alpha)\ d\tau \\
  &\qquad+\l f_2,\phi_n\rd \int_0^t g_2(t-\tau)  E_{\alpha,1} (-\lambda_n \tau^\alpha)\  dB(\tau)\Big]\ \phi_n(x).
  \end{align*}
Again by exchanging the orders of the integral and summation, and noting the representation formula \eqref{eq-v-f}, we can obtain
 \begin{align*}
 I_t^{1-\alpha-\delta} u(x,t,\o)=&\int_0^t  g_1(t-\tau)  v_1(x,\tau)\ d\tau +\int_0^t g_2(t-\tau)  v_2(x,\tau)\ dB(\tau).
  \end{align*}
  We finish the proof of the lemma.
\end{proof}

Now we are ready to give the proof of the uniqueness theorem. 
\begin{proof}[Proof of Theorem \ref{thm-isp}]
Step 1. determine $f_1(x)$. Firstly, in the case of $\alpha+\delta-1\ge0$, from the integral identity in Lemma \ref{lem-duh}, and noting that the assumption $\mathbb E[\partial_{\nu_{\mathcal A}} u(x,t,\o)]=0$ on $\Gamma\times(0,T)$, we can assert that 
\begin{equation}\label{eq-v1}
\int_0^t g_1(\tau) I_t^{\alpha+\delta-1}\partial_{\Nua} v_1(x,t-\tau)\ d\tau =0,\quad (x,t)\in\Gamma\times(0,T).
\end{equation}
In fact, from Lemma \ref{lem-duh} with $\alpha+\delta-1\ge0$, we see that 
\begin{align*}
0=& \partial_{\Nua}u(x,t,\omega) \\
=& \int_0^t g_1( \tau) I_t^{\alpha+\delta-1} \partial_{\Nua} v_1(x,t - \tau) \ d\tau
+\int_0^t g_2(\tau) I_t^{\alpha+\delta-1} \partial_{\Nua} v_2(x, t - \tau) \ dB( \tau)
\end{align*}
for $(x,t,\omega) \in \Gamma\times(0, T)\times \Omega$.
Taking expectation $\mathbb E[\cdot]$ on both sides of the above equation yields
\begin{align*}
0=&\mathbb E \Big[ \int_0^t g_1( \tau) I_t^{\alpha+\delta-1}\partial_{\Nua} v_1(x,t - \tau) \ d\tau \Big]
+ \mathbb E \Big[ \int_0^t g_2(\tau) I_t^{\alpha+\delta-1}\partial_{\Nua} v_2(x, t - \tau) \ dB( \tau) \Big] \\
= & \int_0^t g_1( \tau) I_t^{\alpha+\delta-1}\partial_{\Nua} v_1(x,t - \tau) \ d\tau.
\end{align*}
Here in the last equality, we used the following property of the Ito integral, 
$$
\mathbb E \Big[ \int_0^t g_2(\tau) I_t^{\alpha+\delta-1}\partial_{\Nua} v_2(x, t - \tau) \ d B( \tau) \Big] =0.
$$
Therefore, the assertion \eqref{eq-v1} is true.  Now taking $t$-derivative on both sides of the equation \eqref{eq-v1}, we arrive at
$$
g_1(0) I_t^{\alpha+\delta-1} \partial_{\Nua} v_1(x,t) + \int_0^t g_1'(t - \tau) I_\tau^{\alpha+\delta-1}\partial_{\Nua} v_1(x, \tau) \ d\tau = 0,\quad (x,t)\in\Gamma\times(0,T).
$$
Moreover, since $g_1(0) \ne 0$, we conclude from Gronwall's inequality that
\begin{equation*}\label{eq-I_tv_1}
I_t^{\alpha+\delta-1} \partial_{\Nua} v_1(x,t) = 0, \quad (x,t)\in\Gamma\times(0,T).
\end{equation*}
We claim that $\partial_{\Nua} v_1(x,t)$ must be vanished on $\Gamma\times(0,T)$. In fact, in the case of $\alpha+\delta - 1\ge0$, by noting the semigroup property of Riemann-Liouville integral,   $I_t^{\alpha_1}I_t^{\alpha_2} = I_t^{\alpha_1+\alpha_2}$, we see that 
\begin{align*}
\frac{d}{dt} I_t^{2-\alpha-\delta} I_t^{\alpha+\delta-1} \partial_{\Nua} v_1(x,t) = \partial_{\Nua} v_1(x,t) = 0, \quad (x,t)\in\Gamma\times(0,T).
\end{align*}
For the case of $\alpha+\delta-1<0$, the assertion $\partial_{\Nua} v_1(x,t) = 0$ for $(x,t)\in\Gamma\times(0,T)$ can be proved analogously. 

 Collecting all the above results, we have the following model for $v_1$,  
\begin{equation*}
 \begin{cases}
 \begin{aligned}
  \partial_t^\alpha v_1(x,t) - \Delta v_1(x,t)&= 0, &&\mbox{in }D\times(0,T),\\
v_1(x,0) &= f_1, &&\mbox{in }D,\\
\partial_t v_1(x,0)&=0, &&\mbox{in }D, \mbox{ if } \alpha\in(1,2),\\
v_1 |_{\partial D} = 0,\ \partial_{\Nua} v_1 |_{\Gamma} &= 0, && \mbox{for }t\in(0,T).
 \end{aligned}
\end{cases}
\end{equation*}
From the unique continuation of the fractional diffusion-wave equation, e.g., Theorem 4.2 in Sakamoto and Yamamoto \cite{sakamoto2011initial} and Lemma 2.4 in Liu, Hu and Yamamoto \cite{liuhuyamamoto2020}, we have $v_1=0$ in $D\times(0,T)$, which leads to $f_1= 0$ in $D$.

Step 2. determine $f_2(x)$. We firstly discuss the case of $\alpha+\delta-1\ge0$. Noting that  $f_1\equiv0$ and $\mathrm{Var} [ \partial_{\Nua}u (x,t,\o)] = 0$ on $\Gamma\times(0,T)$, we see that 
\begin{align*}
\mathrm{Var} \big[ \partial_{\Nua}  u(x,t,\o) \big]
= &   \mathbb E \Big[\big(  \partial_{\Nua} u(x,t,\cdot) - \mathbb E[ \partial_{\Nua}  u(x,t,\cdot) ] \big)^2 \Big] \\
=  & \mathbb E \Big[ \big( \int_0^t g_2( \tau) I_t^{\alpha+\delta-1}\partial_{\Nua} v_2(x, t - \tau) \ d B( \tau)\big)^2 \Big] \\
= & \int_0^t g_2^2( \tau) \big| I_t^{\alpha+\delta-1} \partial_{\Nua} v_2(x, t - \tau) \big|^2 \ d\tau = 0
\end{align*}
for $(x,t)\in \Gamma\times(0,T)$, where we use the Ito formula in Lemma \ref{Ito} in the last equality. Therefore by taking $t$-derivative on both sides of the above equality, we analogously get
\begin{align*}
g_2^2 ( 0) \big|I_t^{\alpha+\delta-1}\partial_{\Nua} v_2(x,t) \big|^2+\int_0^t \big[\frac{ d}{ dt} g_2^2( t - \tau)\big]\  \Big[ I_t^{\alpha+\delta-1}\partial_{\Nua} v_2(x,\tau) \Big]^2 \  d\tau = 0 \text{ on }\Gamma \times (0, T).
\end{align*}
Since $g_2(0) \ne 0$, we can employ the Gronwall's inequality to derive that $I_t^{\alpha+\delta-1}\partial_{\Nua} v_2$ is vanished on $\Gamma\times(0,T)$. Consequently, we see that $\partial_{\Nua} v_2 = 0$ on $\Gamma\times(0,T)$ in view of the fact that $\frac{d}{dt} I_t^{2-\alpha-\delta} I_t^{\alpha+\delta-1}$ is the identity operator. Again, by the use of the unique continuation principle for the fractional diffusion-wave equation, we see that $ v_2 = 0$ in $D \times (0,T)$, which immediately implies $f_2 = 0$ in $D$. 

The unique determination of the source $f_2$ in the case of $\alpha+\delta-1<0$ can be analyzed similarly. We therefore finish the proof of the theorem.
\end{proof}

\section{Concluding remark.}\label{sec-con}
In this paper, we discuss the initial-boundary value problem for the stochastic time-fractional diffusion-wave equations. Firstly, the definition of the stochastic weak solution of the forward problem is given, and then the well-posedness of the forward problem is established by using the Fourier expansion argument and the Ito isometry formula. As an application, we consider the inverse problem for determining the multi-factor sources from the moments of the partial boundary measurements. 
The uniqueness of the inverse random source problem is proved by using the unique continuation principle for the fractional diffusion-wave equations.


Next we list two aspects of the future plan. For the forward problems of equation \eqref{equ-gov}, the proof in this work heavily relies on the settings of the source term $F$. The integral operator $I_t^\delta$, $\delta>0$ lifts the regularity of the source term in some sense so that one can get the regularity estimates of the solutions in the usual Lebesgue and Sobolev spaces as in Sakamoto and Yamamoto \cite{Sakamoto2013Inverse}.  
In the case of $\delta=0$, one cannot directly follow the techniques in this work to prove the well-posedness. The theory with $\delta=0$ is one of our future works. For the inverse problems, we see that Theorem \ref{thm-isp} requires the homogeneous Dirichlet boundary condition on the whole boundary. It would be interesting to investigate what will happen if only the lateral Cauchy data on the partial boundary is used. Moreover, in the very recent paper \cite{KST}, the authors considered the inverse problem of identifying a source or an initial state in a       time-fractional diffusion equation from a single boundary measurement on time interval $(0,\infty)$, and the logarithmic stability estimates were established by the use of Laplace inversion techniques and the unique continuation quantification for the resolvent of fractional diffusion operator as a function of the frequency in the complex plane. It will be challenging to investigate the stability of the inverse source problems for the stochastic fractional diffusion-wave equations under the finite time measurements.

\section*{Acknowledgment.}
Matti Lassas is partially supported by Academy of Finland, projects 312339 and 320113. 
Zhiyuan Li thanks National Natural Science Foundation of China (Grant No. 12271277) and Fundamental Research Funds of Ningbo University (Grant No. ZX2022000223). Zhidong Zhang is
supported by National Natural Science Foundation of China (Grant No. 12101627), and the
Fundamental Research Funds for the Central Universities, Sun Yat-sen University (Grant No.
22qntd2901).

\bibliographystyle{abbrv}
\bibliography{sec_ref}

\end{document}